\documentclass{amsart}

\usepackage[T1]{fontenc}
\usepackage{amscd,amssymb,amsfonts}
\usepackage{hyperref}

\usepackage[all]{xy}

\allowdisplaybreaks[2]

\numberwithin{equation}{section}

\newtheorem{dummy}{dummy}[section]
\newtheorem{lemma}[dummy]{Lemma}
\newtheorem{theorem}[dummy]{Theorem}

\newtheorem{proposition}[dummy]{Proposition}

\theoremstyle{definition}
\newtheorem{definition}[dummy]{Definition}

\newtheorem{remark}[dummy]{Remark}

\newcommand{\C}{\mathbb{C}}
\newcommand{\R}{\mathbb{R}}
\newcommand{\Z}{\mathbb{Z}}

\newcommand{\cA}{\mathcal{A}}
\newcommand{\cH}{\mathcal{H}}
\newcommand{\cE}{\mathcal{E}}
\newcommand{\cC}{\mathcal{C}}

\newcommand{\cF}{\mathcal{F}}
\newcommand{\cZ}{\mathcal{Z}}

\newcommand{\cM}{\mathcal{M}}
\newcommand{\cT}{\mathcal{T}}
\newcommand{\cG}{\mathcal{G}}

\newcommand{\Aut}{\mathrm{Aut}}
\newcommand{\Out}{\mathrm{Out}}
\newcommand{\Int}{\mathrm{Int}}

\newcommand{\Id}{\mathrm{Id}}
\newcommand{\Hom}{\mathrm{Hom}}
\newcommand{\Hol}{\mathrm{Hol}}

\newcommand{\irr}{\mathrm{irr}}
\renewcommand{\deg}{\mathrm{deg}}

\newcommand{\GL}{\mathbf{GL}}
\newcommand{\PGL}{\mathbf{PGL}}
\newcommand{\U}{\mathbf{U}}
\newcommand{\PSL}{\mathbf{PSL}}

\newcommand{\si}{\sigma}
\newcommand{\Si}{\Sigma}
\newcommand{\ga}{\gamma}
\newcommand{\Ga}{\Gamma}
\newcommand{\lra}{\longrightarrow}
\newcommand{\lmt}{\longmapsto}
\newcommand{\eps}{\varepsilon}
\newcommand{\piX}{\pi_1(X;x)}
\newcommand{\Xt}{\widetilde{X}}
\newcommand{\Css}{\cC_{ss}}
\newcommand{\Cs}{\cC_{s}}
\newcommand{\Cmu}{\cC_{\mu}}
\newcommand{\dmu}{d_{\mu}}
\newcommand{\codim}{\mathrm{codim}}
\newcommand{\piorb}{\pi_1^{\mathrm{orb}}}
\newcommand{\taut}{\widetilde{\tau}}

\newcommand{\ov}[1]{\overline{#1}}

\renewcommand{\phi}{\varphi}
\renewcommand{\rho}{\varrho}

\begin{document}

\title[Finite group actions on moduli spaces of vector bundles]{Finite group actions on moduli spaces of\\ vector bundles}
\author{Florent Schaffhauser}
\address{Departamento de Matem\'aticas, Universidad de Los Andes, Bogot\'a, Colombia.}
\email{florent@uniandes.edu.co}

\subjclass[2000]{14H60,14H30}

\keywords{Vector bundles on curves and their moduli,Fundamental groups}

\date{\today}

\thanks{The author acknowledges the support of U.S. National Science Foundation grants DMS 1107452, 1107263, 1107367 "RNMS: Geometric structures And Representation varieties" (the GEAR Network). Thanks go to Victoria Hoskins for helpful discussions on quiver varieties analogues of the topics covered in this paper.}

\begin{abstract}
We study actions of finite groups on moduli spaces of stable holomorphic vector bundles and relate the fixed-point sets of those actions to representation varieties of orbifold fundamental groups.
\end{abstract}

\maketitle


\tableofcontents

Let $X$ be a compact connected Riemann surface of genus $g\geq 2$. The goal of the paper is to study certain finite group actions on moduli spaces of stable vector bundles over $X$, equipped with their standard complex analytic structure (\cite{NS1,NS}), where by automorphism of such a moduli space, we mean a transformation that is either holomorphic or anti-holomorphic. 

In Section \ref{autom}, we show how such actions arise naturally from groups of holomorphic and anti-holomorphic transformations of $X$ and we describe these actions from three different points of view: algebro-geometric, gauge-theoretic, and via the fundamental group of $X$. Then, in Section \ref{modular_interp_section}, we dig deeper into the gauge-theoretic picture in order to interpret the fixed points of the actions constructed earlier in terms of moduli of certain vector bundles with extra structure. Finally, in Section \ref{NS_and_orb}, we relate those vector bundles with extra structure to unitary representations of orbifold fundamental groups.

\section{Automorphisms of the moduli space}\label{autom}

\subsection{The algebro-geometric picture}\label{alg_geo_pic}

Let $X$ be a compact connected Riemann surface of genus $g\geq 2$. A holomorphic vector bundle $\cE$ over $X$ is called \textit{stable} if, for all non-trivial holomorphic sub-bundle $\cF\subset\cE$, one has $\deg(\cF^*\otimes\cE) >0$ (semistability is defined using a large inequality instead). Rank $r$ vector bundles are topologically classified by their degree $d\in H^2(X;\pi_1(\GL(r;\C)))\simeq \pi_1(\GL(r;\C)) \simeq \Z$ and, by a theorem due to Mumford (\cite{Mumford_Proc}), the set $\cM_X(r,d)$, consisting of isomorphism classes of stable holomorphic vector bundles and degree $d$ admits a natural structure of non-singular quasi-projective variety of dimension $r^2(g-1)+1$. An important property of stable holomorphic vector bundles is that they are \textit{simple} (they only admit those automorphisms that come from the centre of the structure group).
%
%
As a matter of fact, much of what we shall say in this paper extends to moduli spaces of simple and stable principal holomorphic $G_\C$-bundles of fixed topological type $d\in H^2(X;\pi_1(G_\C)) \simeq \pi_1(G_\C)$, for $G_\C$ a (connected) reductive complex Lie group. Such moduli spaces exist (\cite{Ramanathan}) and have dimension $\dim G_\C\,(g-1) + \dim \cZ(G_\C)$, where $\cZ(G_\C)$ is the centre of $G_\C$. This suggests adopting the following notation for the remainder of this section: $G_\C:=\GL(r;\C)$ and $\cM_X(G_\C,d) :=\cM_X(r,d)$. 

Our goal is now to define finite group actions on $\cM_X(G_\C,d)$. These could in principle be induced by two different types of transformations:

\begin{enumerate}
\item either holomorphic or anti-holomorphic transformations of the complex Lie group $G_\C$;
\item either holomorphic or anti-holomorphic transformations of the Riemann surface $X$.
\end{enumerate}

Indeed, if $\cE$ is a bundle represented by a holomorphic $1$-cocycle $g_{ij}: U_i\cap U_j \lra G_\C$, $u\in \Aut(G_\C)$ is a group automorphism and $\si:X\lra X$ is an automorphism of $X$, we can form the $1$-cocycle $u\circ g_{ij} \circ \si^{-1}:\si(U_i)\cap\si(U_j) \lra \GL(r;\C)$. The issue is that we need to guarantee that $u\circ g_{ij}\circ \si^{-1}$ is still holomorphic, so we have to impose that $u$ and $\si$ be both holomorphic or both anti-holomorphic. In the present paper, we will achieve that for the following class of examples. Let $\Aut^\pm(X)$ be the group of either holomorphic or anti-holomorphic bijective transformations of $X$. Then there is a short exact sequence
\begin{equation}\label{ses_Aut_X}
1\lra \Aut(X) \lra \Aut^\pm(X) \lra \Z/2\Z\lra 1
\end{equation}
\noindent whose kernel consists of holomorphic automorphisms of $X$ (a splitting of \eqref{ses_Aut_X} is equivalent to an anti-holomorphic involution of $X$; it may or not exist). Since $g\geq 2$ by assumption, $\Aut^\pm(X)$ is a finite group of order no greater than $168(g-1)$. Let $\Si$ be a subgroup of $\Aut^\pm(X)$ and denote by $\eps:\Si\lra\Z/2\Z$ the group homomorphism induced by \eqref{ses_Aut_X}: \begin{equation}\label{construction_eps}\eps(\si) = \left\{\begin{array}{cl}
+1 & \mathrm{if}\ \si:X\lra X\ \mathrm{is\ holomorphic,}\\
-1 & \mathrm{if}\ \si:X\lra X\ \mathrm{is\ anti}-\mathrm{holomorphic.}
\end{array}\right.\end{equation}

Let now $\theta:G_\C\lra G_\C$ be an anti-holomorphic, involutive group homomorphism. This induces an action of $\Si$ on $G_\C$, defined by \begin{equation}\label{Si_action_on_G_def}\theta_\si:= \left\{ \begin{array}{cl}
\mathrm{Id}_{G_\C} & \mathrm{if}\ \eps(\si)=+1,\\
\theta & \mathrm{if}\ \eps(\si)=-1.
\end{array}\right.\end{equation} Thus, for all $\si\in \Si$ and $(g_{ij})_{(i,j)}$ a holomorphic $1$-cocycle on $X$ representing a $G_\C$-bundle $\cE$, we can define $\si(\cE)$ to be the $G_\C$-bundle defined by the holomorphic $1$-cocycle \begin{equation}\label{action_on_cocyclees}\si(g_{ij}):= \theta_\si\circ g_{ij}\circ \si^{-1}.\end{equation}

The real form $\theta$ of $G_\C=\GL(r;\C)$ that will be most important to us is $\theta:g\lra\ov{g}$, because then the standard action of $G_\C=\GL(r;\C)$ on $\C^r$ is a real action: $\ov{gv}=\ov{g}\,\ov{v}$. Then we have:  \begin{equation}\label{Si_action_on_bdles}\si(g_{ij})=\left\{\begin{array}{cl}
g_{ij}\circ\si^{-1} & \mathrm{if}\ \eps(\si)=+1,\\
\ov{g_{ij}\circ\si^{-1}} & \mathrm{if}\ \eps(\si)=-1.
\end{array}\right.
\end{equation}
\noindent In particular, if $(g_{ij})_{(i,j)}$ is a cocycle representing the vector bundle $\cE$, then the vector bundle $\si(\cE)$ defined by $\si(g_{ij})$ is isomorphic to $(\si^{-1})^*\cE$ if $\si$ is holomorphic and to $\ov{(\si^{-1})^*\cE}$ is $\si$ is anti-holomorphic, which readily implies the next result.

\begin{proposition}\label{Si_action_on_mod_spaces}
Let $\cE$ be a holomorphic vector bundle on $X$ and consider $\si\in\Si\subset\Aut^\pm(X)$. Then the holomorphic vector bundle $\si(\cE)$ defined by the $\Si$-action \eqref{Si_action_on_bdles} is stable if and only if $\cE$ is stable Moreover, there is an induced a $\Si$-action on each of the moduli spaces $\cM_X(r,d)$, for all $(r,d)$.
\end{proposition}
\begin{proof}
For holomorphic vector bundles, both properties follow from the fact that $\deg(\si(\cE))=\deg\,\cE$ when $\Si$ acts as described in \eqref{Si_action_on_bdles}.
\end{proof}

Note that it is not true in general that there is an induced $\Si$-action on every moduli space $\cM_X(G_\C,d)$. For instance, if $\theta(g)=(\ov{g}^t)^{-1}$ on $\GL(r;\C)$, then $\deg(\si(\cE))=-\deg\,\cE$, so only the moduli space $\cM_X(G_\C,0)$ will be preserved by the induced $\Si$-action.

\subsection{The gauge-theoretic picture}\label{gauge_theoretic_picture}

The Narasimhan-Seshadri theorem (\cite{NS,Don_NS}) provides a diffeomorphism between the moduli space $\cM_X(r,d)$ of stable holomorphic vector bundles of rank $r$ and degree $d$ over $X$ and the space, to be denoted by $F^{-1}(\{\frac{d}{r}\})_{\irr}/\cG_E$, constisting of gauge equivalence classes of projectively flat, irreducible unitary connections on a fixed Hermitian vector bundle $E$ over $X$, of rank $r$ and degree $d$. Here, by $\cG_E$, we denote the unitary gauge group of the Hermitian vector bundle $E$ and by $F$ the map taking a unitary connection $A$ on $E$ to its curvature form $F(A)$, which is a differential $2$-form on $E$ with values in the bundle of anti-Hermitian endomorphisms of $E$. If we use a compatible Riemannian metric of volume $2\pi$ on $X$ and replace anti-Hermitian endomorphisms by Hermitian ones (using multiplication by $\sqrt{-1}$), we have indeed that projectively flat  connections are those connections that satisfy $F(A)=\frac{d}{r}\Id_E$, hence our notation above. A unitary connection $A$ on $E$ is called \textit{irreducible} if there is no non-trivial sub-bundle $F\subset E$ whose sections are preserved by the covariant derivative associated to $A$. Equivalently, there is no direct sum decomposition $E\simeq E_1\oplus E_2$ such that $A=A_1\oplus A_2$. Or yet again, the stabiliser of $A$ in $\cG_E$ is equal to the centre $\cZ(\cG_E) \simeq \cZ(\U(r)) \simeq S_1$ of the gauge group $\cG_E$.

Let us now assume that the $\Si$-action on $X$ lifts to $E$, i.e.\ that there exists a family $(\tau_\si)_{\si\in\Si}$ of (normalised) bundle isometries from $E$ to itself satisfying the conditions:

\begin{enumerate}
\item For all $\si\in\Si$, we have a commutative diagram $$\begin{CD}
E @>{\tau_\si}>> E \\
@VVV @VVV\\
X @>{\si}>> X
\end{CD}$$
\item If $\si:X\lra X$ is holomorphic, the map $\tau_\si:E\lra E$ is fibrewise $\C$-linear. If $\si:\lra X$ is anti-holomorphic,  the map $\tau_\si:E\lra E$ is fibrewise $\C$-anti-linear.
\item $\tau_{1_\Si}=\Id_E$ and, for all $(\si_1,\si_2)\in\Si\times \Si$, $\tau_{\si_1\si_2} = \tau_{\si_1}\tau_{\si_2}$. In particular, for all $\si$, $\tau_\si$ is invertible and $\tau_\si^{-1} = \tau_{\si^{-1}}$.
\end{enumerate}

This is equivalent to asking that there exists a family $(\phi_\si)_{\si\in E}$ of (normalised) isomorphisms satisfying the conditions:

\begin{enumerate}
\item For all $\si\in\Si$, we have a commutative diagram $$\xymatrix{
\si(E) \ar[rr]^{\phi_\si} \ar[dr] & & E \ar[dl]\\
&X &}$$
\item The map $\phi_\si:\si(E)\lra E$ is fibrewise $\C$-linear.
\item $\phi_{1_{\si}} = \Id_{E}$ and, for all $(\si_1,\si_2)\in\Si\times\Si$, $\phi_{\si_1\si_2} = \phi_{\si_1}\si_1(\phi_{\si_2})$. In particular, for all $\si$, $\phi_{\si}$ is an isomorphism and $\phi_{\si}^{-1} = \si(\phi_{\si^{-1}})$.
\end{enumerate}

Given an element $\si\in \Si$ and a unitary connection $A$ on $E$, we denote by $\si(A)$ the induced connection on $\si(E)$, where $\si(E)$ is the Hermitian vector bundle defined as in \eqref{action_on_cocyclees}, except that we are now working in the $C^\infty$ category. Using the isomorphism $\phi_{\si}:\si(E)\lra E$, we can then define \begin{equation}\label{Si_action_on_conn_def}
\beta_{\si}(A) := (\phi_{\si}^{-1})^*\big(\si(A)\big).
\end{equation}

\begin{proposition}\label{Si_action_on_conn}
The map $\si\lmt\beta_{\si}$ induced by \eqref{Si_action_on_conn_def} defines a $\Si$-action on the space $\cA_E$ of unitary connections on $E$.
\end{proposition}

\begin{proof}
One has, for all $(\si_1,\si_2)\in\Si\times\Si$, \begin{eqnarray*}
\beta^{\tau}_{\si_1\si_2}(A) & = & (\phi_{\si_1\si_2}^{-1})^*\big((\si_1\si_2)(A)\big)\\
& = & \big(\si_1(\phi_{\si_2}^{-1})\phi_{\si_1}^{-1}\big)^*\big(\si_1(\si_2(A))\big) \\
& = & (\phi_{\si_1}^{-1})^* \big(\si_1((\phi_{\si_2}^{-1})^*\si_2(A))\big)\\
& = & (\beta^{\tau}_{\si_1}\circ\beta^{\tau}_{\si_2})(A),
\end{eqnarray*} which proves the proposition.
\end{proof}

Since $\cZ(\cG_E)\simeq\cZ(\U(r))$ acts trivially on $\cA_E$, we can actually relax Condition (3) in the definitions of the families $(\tau_\si)_{\si\in\Si}$ and $(\phi_\si)_{\si\in\Si}$ and still obtain Proposition \ref{Si_action_on_conn}. Indeed, if $c:\Si\times\Si\lra\cZ(\U(r))$ is any (normalised) $2$-cocycle, one need only ask 
\begin{itemize}
\item[$\mathrm{(3')}$] $\tau_{\si_1}\tau_{\si_2} = c(\si_1,\si_2)\tau_{\si_1\si_2}$
\end{itemize}
\noindent or equivalently, $\phi_{\si_1}\si_1(\phi_{\si_2}) = c(\si_1,\si_2) \phi_{\si_1\si_2}$, because then, in the proof of Proposition \ref{Si_action_on_conn}, one obtains $$\beta^{\tau}_{\si_1\si_2}(A) = \big(c(\si_1,\si_2)\big)^* \big(\beta^{\tau}_{\si_1}\circ\beta^{\tau}_{\si_2}\big) (A) = (\beta^{\tau}_{\si_1}\circ\beta^{\tau}_{\si_2})(A),$$ since $c(\si_1,\si_2)\in\cZ(\U(r))$ and $\cZ(\U(r))$ acts trivially on $\cA_E$. This motivates the following definition.

\begin{definition}\label{c-twisted_lift}
Given a $2$-cocycle $c:\Si\times \Si\lra \cZ(\U(r))$, a $c$-equivariant structure on the Hermitian vector bundle $E$ is a family $(\tau_\si)_{\si\in\Si}$ of maps $\tau_\si:E\lra E$ satisfying Conditions (1), (2) and (3') above.
\end{definition}

\noindent When $c$ is the trivial cocycle, a $c$-equivariant structure will be called just an equivariant structure (the resulting notion of equivariant bundle was already studied in \cite{Andersen-Grove}, for the case of holomorphic transformations on $X$). The issue is then to check that the $\Si$-action on $\cA_E$ defined in \eqref{Si_action_on_conn_def}, which depends on the $c$-equivariant structure $(\tau_{\si})_{\si\in\Si}$ on $E$, induces a $\Si$-action on $F^{-1}(\{\frac{d}{r}\})_\irr/\cG_E$. When $d=0$, this is a consequence of the following general framework for $(\Si,\cG_E)$-compatible group actions on Hamiltonian spaces:
\begin{enumerate}
\item There is a $\Si$-action on $\cG_E$, defined for $\si\in\Si$ by \begin{equation}\label{Si_action_on_gauge_gp_def}
\alpha^{\tau}_\si(g)=\tau_\si g\tau_\si^{-1} = \big(\phi_{\si}^{-1}\big)^* \big(\si(g)\big),\end{equation} which is compatible with the $\Si$-action on $\cA_E$ defined in \eqref{Si_action_on_conn_def} in the sense that \begin{equation}\label{comp_condition_group_actions}\beta^{\tau}_\si(g\cdot A) = \alpha^{\tau}_\si(g)\cdot\beta^{\tau}_\si(A).\end{equation} Note that, in order to define the $\Si$-action on the gauge group $\cG_E$ by $\alpha^{\tau}_\si$ by the expression \eqref{Si_action_on_gauge_gp_def}, one implicitly uses the $\Si$-action on the structure group $G_\C$ defined in \eqref{Si_action_on_G_def}: the real form $\theta:G_\C\lra G_\C$ is hidden in the notation $\si(g)$ for the gauge transformation of $\si(E)$ induced by the gauge transformation $g$ of $E$.
\item The $\Si$-action on $\cG_E$ defined above induces a $\Si$-action on the dual of the Lie algebra of $\cG_E$, which is where the ($\cG_E$-equivariant) momentum map $F$ for the $\cG_E$-action on $\cA_E$ takes its values (\cite{AB}). If we still denote by $\alpha^{\tau}_\si$ this induced action, we have the following compatibility relation between the $\Si$-action on $\cA_E$ and the momentum map $F$: $$F\big(\beta^{\tau}_\si(A)\big) = \alpha^{\tau}_\si\big(F(A)\big).$$
\end{enumerate}

\begin{proposition}\label{Si_action_on_Hamil_quotient}
The $\Si$-action on $\cA_E$ defined in \eqref{Si_action_on_conn_def} induces a $\Si$-action on the Hamiltonian reduction $F^{-1}(\{0\})_\irr/\cG_E$.
\end{proposition}

\begin{proof}
By Condition (1), the stabiliser of $\beta^{\tau}_\si(A)$ in $\cA_E$ is equal to $\alpha^{\tau}_\si(\mathrm{Stab}_{\cG_E}(A))$. Since $\Si$ acts on $\cG_E$ by group automorphisms, it preserves the centre of $\cG_E$ and, consequently, the space of irreducible connections. By Condition (2), if $A$ is a flat unitary connection, the curvature of $\beta^{\tau}_\si(A)$ is $\alpha^{\tau}_\si(F(A))$, which is $0$ because $\Si$ acts linearly on the Lie algebra of $\cG_E$. In particular, the $\Si$-action preserves the space of flat connections in $\cA_E$. Condition (1) then ensures that the $\Si$-action on $F^{-1}(\{0\})_\irr = F^{-1}(\{0\})\cap(\cA_E)_\irr$ descends to $F^{-1}(\{0\})_\irr/\cG_E$.
\end{proof}

To see that Proposition \ref{Si_action_on_Hamil_quotient} remains true when $d\neq 0$, we need to use special properties of the involution $\theta$ on $G_\C$: since we have chosen $\theta$ to be the split real form of $G_\C=\GL(r;\C)$, the induced involution on the Lie algebra of $\mathfrak{u}(r)$ restricts to $\xi\lmt -\xi$ on the centre of $\mathfrak{u}(r)$. Therefore, the induced $\Si$-action on $Lie(\cG_E)$ fixes the element $i\frac{d}{r}\Id_E$ and we can shift the momentum map by such an element and reduce the $d\neq 0$ case to the $d=0$ case. This shows, as in Proposition \ref{Si_action_on_mod_spaces}, the importance of working with the split real form $\theta$ of $G_\C$ in order to obtain a $\Si$-action on each of the moduli spaces $\cM_X(G_\C,d)$, for all $d\in\pi_1(G_\C)$.

To conclude the present subsection, recall that unitary connections on $E$ are in bijective correspondence with holomorphic structures on that bundle. More precisely, to each unitary connection $A$ on $E$ there is associated a Dolbeault operator $\ov{\partial}_A$ on $E$ whose kernel defines a locally free sheaf of rank $r$ over the sheaf of holomorphic functions on $X$, which in turn defines a holomorphic vector bundle $\cE$. It can then be checked (see for instance \cite{Sch_JSG} for the case where $\Si\simeq\Z/2\Z$ acts on $X$ via an anti-holomorphic involution $\si$) that the holomorphic vector bundle $\si(\cE)$ defined in \eqref{Si_action_on_bdles} is the holomorphic vector bundle associated to the unitary connection $\beta^{\tau}_\si(A)$ defined in \eqref{Si_action_on_conn_def}. Since $\si(\cE)$ does not depend on the choice of a background Hermitian bundle $E$ with a $c$-equivariant structure $(\tau_\si)_{\si\in\Si}$, the $\Si$-action on $F^{-1}(\{0\})_\irr/\cG_E$ defined in Proposition \ref{Si_action_on_Hamil_quotient} \textit{always} corresponds, under the Narasimhan-Seshadri diffeomorphism $F^{-1}(\{0\})_\irr/\cG_E \simeq \cM_X(r,0)$ to the canonical $\Si$-action $\cE\lmt\si(\cE)$. In particular, the $\Si$-action induced on the Hamitlonian quotient $F^{-1}(\{0\})_\irr/\cG_E$ is independent of the choice of the $c$-equivariant structure $(\tau_\si)_{\si\in\Si}$ and also of the choice of the cocycle $c$, as we shall now prove directly.

\begin{proposition}\label{induced_action_is_canonical}
We have the following results:
\begin{enumerate}
\item Fix a $2$-cocycle $c\in Z^2(\Si;\cZ(\U(r)))$ and let $(\phi_\si)_{\si\in\Si}$ and $(\psi_\si)_{\si\in\Si}$ be any two $c$-equivariant structures on $E$, the two induced $\Si$-actions on $\cA_E$ being respectively denoted by $(\beta^{\tau}_\si)_{\si\in\Si}$ and $(\ga_\si)_{\si\in\Si}$. Then $(\beta^{\tau}_\si)_{\si\in\Si}$ and $(\ga_\si)_{\si\in\Si}$ induce the same $\Si$-action on $F^{-1}(\{\frac{d}{r}\})/\cG_E$.
\item Let $c$ and $c'$ be two $2$-cocycles $\Si\times\Si\lra \cZ(\U(r))$. Let $(\phi_\si)_{\si\in\Si}$ be a $c$-equivariant structure on $E$ and let $(\phi'_\si)_{\si\in\Si}$ be a $c'$-equivariant structure on $E$. The two induced $\Si$-actions on $\cA_E$ are respectively denoted by $(\beta^{\tau}_\si)_{\si\in\Si}$ and $(\beta^{\tau,u}_\si)_{\si\in\Si}$. Then $(\beta^{\tau}_\si)_{\si\in\Si}$ and $(\beta^{\tau,u}_\si)_{\si\in\Si}$ induce the same $\Si$-action on $F^{-1}(\{\frac{d}{r}\})/\cG_E$.
\end{enumerate}
\end{proposition}

\begin{proof} Recall that we have denoted by $(\alpha^{\tau}_\si)_{\si\in\Si}$ the $\Si$-action on $\cG_E$ induced by the the choices of a cocycle $c$ and a $c$-equivariant structure $(\phi_\si)_{\si\in\Si}$: for all $\si\in\Si$ and all $g\in \cG_E$, $\alpha^{\tau}_\si(g) = \phi_\si \si(g) \phi_\si^{-1}$. Recall also the compatibility relation \eqref{comp_condition_group_actions} between $\alpha^{\tau}_\si$ and $\beta^{\tau}_\si$. 
\begin{enumerate}
\item Given $\si\in\Si$, let us consider the gauge transformation $u_\si:=\psi_\si\phi_\si^{-1}\in\cG_E$. Then one has, for all $\si\in\Si$ and all $A\in\cA_E$,
$$\ga_\si(A) = (\psi_\si^{-1})^* \si(A) = (\phi_\si^{-1} u_\si^{-1})^* \si(A) = (u_\si^{-1})^* \big((\phi_\si^{-1})^*\si(A)\big) = u_\si \cdot \beta^{\tau}_\si(A),$$ so $(\beta^{\tau}_\si)_{\si\in\Si}$ and $(\ga_\si)_{\si\in\Si}$ indeed induce the same $\Si$-action on $F^{-1}(\{\frac{d}{r}\})/\cG_E$.
\item The proof is the same as (1), setting now $u_\si:=\phi'_\si\phi_\si^{-1}\in\cG_E$.  
\end{enumerate}
\end{proof}

\noindent Note that a difference between (1) and (2) in the above proof is that in (1) we have $u_{\si_1\si_2} = u_{\si_1}\alpha^{\tau}_{\si_1}(u_{\si_2})$ (as seen by a direct computation), while in (2) we have $u_{\si_1\si_2} = c(\si_1,\si_2) c'(\si_1,\si_2)^{-1} u_{\si_1}\alpha^{\tau}_{\si_1}(u_{\si_2})$. In particular, if the cocycles $c$ and $c'$ are cohomologous, i.e.\ if there exists a (normalised) map $a:\Si\lra \cZ(\U(r))$ such that, for all $(\si_1,\si_2)\in\Si\times\Si$, $c(\si_1,\si_2)c'(\si_1,\si_2)^{-1}= a_{\si_1}\si_1(a_{\si_2}) a_{\si_1\si_2}^{-1}$, then the family $(h_\si:=a_\si u_\si)_{\si\in\Si}$ satisfies $h_{\si_1\si_2} = h_{\si_1}\alpha^{\tau}_{\si_1}(h_{\si_2})$, meaning that in this case we can actually modify the $(g_{\si})_{\si\in\Si}$ for the cocycle relation to hold, without modifying the way these gauge transformations act on $\cA_E$.

\subsection{Representations of the fundamental group}\label{rep_fund_gp_section}

Let $x$ be a point in $X$ and let $\piX$ be the fundamental group of $X$ at that point. For stable holomorphic vector bundles of degree $0$, the diffeomorphism $\cM_X(r,0) \simeq F^{-1}(\{0\})_\irr/\cG_E$ also reads $$\cM_X(r,0)\simeq \Hom(\piX;\U(r))_\irr / \U(r)\,.$$ Explicitly, given a representation $\rho$ of $\piX$ into the maximal compact subgroup $K:=\U(r)$ of $G_\C=\GL(r;\C)$, there is a natural $\piX$-action on the product bundle $\Xt_x\times\C^r$ of the universal covering space $\Xt_x$ of $X$ (canonically determined by the choice of the base point $x$) with the vector space $\C^r$: for all $\ga\in\piX$, all $\xi\in \Xt_x$ and all $v\in\C^r$, $$\ga\cdot (\xi,v) := (\ga\cdot\xi,\rho(\ga)\cdot v),$$ where $\piX$ acts on $\piX$ in the canonical way and on $\C^r$ via the representation $\rho$. If $\rho:\piX\lra \U(r)$ is irreducible (meaning that there is no non-trivial sub-space $W\subset\C^r$ invariant under all transformations in $\mathrm{Im}\,\rho$, or equivalently that $\rho$ is not the direct sum of two unitary representations, or yet again that the stabiliser of $\rho$ in $\U(r)$ is reduced to $\cZ(\U(r))\simeq S^1$), then the induced rank $r$ vector bundle $$\cE_{\rho}:=(\Xt_x\times\C^r)/ \piX$$ over $X$ is stable, of degree $0$. Conversely, any such bundle comes in that way from an irreducible unitary representation of $\piX$ (\cite{NS}). 

Recall from Section \ref{alg_geo_pic} that $\Si\subset\Aut^\pm(X)$ is a finite group which also acts on $G_\C$ via the group homomorphism $\eps:\Si\lra \Z/2\Z$ constructed in \eqref{construction_eps} and the anti-holomorphic involution $\theta$ of $G_\C$. Conjugating $\theta$ by an inner automorphism of $G_\C$ if necessary, we can assume that $\theta$ preserves the maximal compact subgroup $K$ of $G_\C$. Therefore, there is an induced group homomorphism $$\Si\lra \Out(K)\times \Out(\piX),$$ where $\Out(H) := \Aut(H)/\Int(H)$ denotes the outer automorphism group of a given group $H$. Note that, since $X$ is compact connected of genus $g\geq 2$, we can identify $\Out(\piX)$ with the mapping class group of $X$ and then the homomorphism $\Si\lra\Out(\piX)$ just takes a transformation $\si:X\lra X$ to its mapping class. Since moreover $\Out(K)\times \Out(\piX)$ acts on the representation space $\Hom(\piX;K)_\irr/K$ via $$([\psi],[f])\cdot [\rho] := [\psi\circ \rho\circ f^{-1}],$$ where we identify $\pi_1(X;f(x))$ and $\piX$ via the choice of a path from $x$ to $f(x)$. Consequently, there is an induced $\Si$-action on the representation space $\Hom(\piX;K)_\irr /K$, defined by $$\si\cdot[\rho] := [\theta_{\si}\circ\rho\circ \si^{-1}]=:[\si(\rho)].$$ In particular, the holomorphic vector bundle $\cE_{\si(\rho)}$ is isomorphic to $\si(\cE_\rho)$.

To sum up, we have three ways of thinking about the moduli space $\cM_X(r,0)$:
\begin{enumerate}
\item the space of isomorphism classes of stable holomorphic vector bundles of rank $r$ and degree $0$ over $X$;
\item the space of gauge orbits of irreducible unitary connections on a fixed Hermitian vector bundle of rank $r$ and degree $0$ over $X$;
\item the space of irreducible representations $\piX\lra \U(r)$.
\end{enumerate}
\noindent Then, given a finite group $\Si\subset\Aut^\pm(X)$ of either holomorphic or anti-holomorphic transformations of $X$, there is an induced $\Si$-action on $\cM_X(r,0)$ that we can describe in each of the three pictures above (which amounts to saying that the known diffeomorphisms between those three spaces are $\Si$-equivariant). Perhaps noteworthy is the fact that it is slightly more involved to define the induced $\Si$-action on the space $F^{-1}(\{0\})_\irr /\cG_E$ than in the other two models of $\cM_X(r,0)$, because it involves a choice of an extra structure, namely a $c$-equivariant structure on $E$ (Definition \ref{c-twisted_lift}) in order to define the induced action on the moduli space, making it necessary to check that this action is independent of that choice. This might be seen as a drawback at first but, as we shall see in the next two sections, it holds the key to two important things:
\begin{enumerate}
\item the modular interpretation of $\Si$-fixed points in $\cM_X(r,0)$;
\item the Narasimhan-Seshadri correspondence in this context.
\end{enumerate}
 A further application of the gauge-theoretic point of view would be the study of the topology of that fixed-point set 
but we shall not pursue this here.
 
\section{Modular interpretation of the fixed-point set}\label{modular_interp_section}

\subsection{Equivariant structures on holomorphic vector bundles}

As mentioned at the end of Section \ref{rep_fund_gp_section}, the gauge-theoretic perspective on finite group actions on moduli spaces of stable holomorphic vector bundles holds the key to the modular interpretation of the fixed points of such an action. Indeed, we will now equip holomorphic vector bundles with generalised equivariant structures that mirror, in the holomorphic category, the $c$-equivariant structures introduced in Definition \ref{c-twisted_lift}. 

\begin{definition}\label{c-equiv_struct}
Let $X$ be a Riemann surface equipped with a finite group action of $\Si\subset \Aut^\pm(X)$. Let $\cE$ be a holomorphic vector bundle over $X$ and fix a group $2$-cocycle $c:\Si\times \Si\lra \C^*\subset\Aut(\cE)$. A $c$-equivariant structure on $\cE$ is a family $(\tau_\si)_{\si\in\Si}$ of bundle maps from $\cE$ to itself satisfying the conditions:
\begin{enumerate}
\item For all $\si\in\Si$, we have a commutative diagram $$\begin{CD}
E @>{\tau_\si}>> E \\
@VVV @VVV\\
X @>{\si}>> X
\end{CD}$$
\item If $\si:X\lra X$ is holomorphic, the map $\tau_\si:E\lra E$ is holomorphic and fibrewise $\C$-linear. If $\si:\lra X$ is anti-holomorphic,  the map $\tau_\si:E\lra E$ is anti-holomorphic and fibrewise $\C$-anti-linear.
\item $\tau_{1_\Si}=\Id_E$ and, for all $(\si_1,\si_2)\in\Si\times \Si$, $\tau_{\si_1}\tau_{\si_2} = c(\si_1,\si_2) \tau_{\si_1\si_2}$.
\end{enumerate}
A pair $(\cE,(\tau_\si)_{\si\in\Si})$ where $\cE$ is a holomorphic vector bundle and $(\tau_\si)_{\si\in\Si}$ is a $c$-equivariant structure on $\cE$ will be called a $c$-equivariant holomorphic vector bundle. A homomorphism between two $c$-equivariant vector bundles $(\cE,(\tau_\si)_{\si\in\Si})$ and $(\cE',(\tau'_\si)_{\si\in\Si})$ (for the same cocycle $c$) is a homomorphism of holomorphic vector bundles $f:\cE\lra \cE'$ over $X$ such that, for all $\si\in \Si$, $f\circ\tau_\si= \tau'_\si\circ f$.
\end{definition}

Recall that when $c$ is the trivial cocycle, a $c$-equivariant structure is just a lifting of the $\Si$-action on $X$ to a $\Si$-action on $\cE$. Also, Definition \ref{c-equiv_struct} is equivalent to asking that there exist a family $(\phi_\si)_{\si\in \cE}$ of (normalised) isomorphisms satisfying the conditions:

\begin{enumerate}
\item For all $\si\in\Si$, we have a commutative diagram $$\xymatrix{
\si(\cE) \ar[rr]^{\phi_\si} \ar[dr] & & \cE \ar[dl]\\
&X &}$$
\item The map $\phi_\si:\si(\cE)\lra \cE$ is holomorphic and fibrewise $\C$-linear.
\item $\phi_{1_{\si}} = \Id_{\cE}$ and, for all $(\si_1,\si_2)\in\Si\times\Si$, $\phi_{\si_1}\si_1(\phi_{\si_2}) = c(\si_1,\si_2)\phi_{\si_1\si_2}$.
\end{enumerate}
\noindent When $c$ is the trivial cocycle, this should really be thought of as a \textit{descent datum}: it is a necessary and sufficient condition for the holomorphic vector bundle $\cE$ to descend to a vector bundle in the appropriate category over the orbifold $[X/\Si]$.

Stability of a $c$-equivariant holomorphic vector bundle $(\cE,\tau:=(\tau_\si)_{\si\in\Si})$ may be defined in the usual way, with respect to non-trivial $\tau$-invariant sub-bundles only. This is in general not equivalent to the underlying holomorphic vector bundle $\cE$ being stable (see \cite{Ramanan_hyperelliptic} for the case of a holomorphic involution and \cite{Sch_JSG} for the case of an anti-holomorphic one). But this is actually a good thing because it gives us enough stable objects for the Seshadri theorem (\cite{Seshadri}) to hold: any semi-stable object admits filtrations by $\tau$-invariants sub-objects whose successive quotients are stable (in the weaker, $\tau$-relative, sense, but in general not in the strong sense of the underlying bundle being stable) and of equal slope. This is however not something that we will pursue in this paper (see \cite{Sch_JSG} for the special case of anti-holomorphic involutions). Instead, as in \cite{HN}, we shall focus on the following family of objects.

\begin{definition}
A $c$-equivariant holomorphic vector bundle $(\cE,(\tau_\si)_{\si\in\Si})$ is called \emph{geometrically stable} if the underlying vector bundle $\cE$ is stable in the usual sense. It is called \emph{polystable} if it is a direct sum of geometrically stable $c$-equivariant holomorphic vector bundles of equal slope.
\end{definition}

The rest of Section \ref{modular_interp_section} is devoted to proving the existence of geometrically stable $c$-equivariant vector bundles and constructing connected (gauge-theoretic) moduli spaces for such objects.

\subsection{Invariant connections}\label{inv_conn_section}

We saw in Section \ref{gauge_theoretic_picture} that, given a $c$-equivariant Hermitian vector bundle $E$ over $(X,\Si)$, there was a $\Si$-action on the space $\cA_E$ of unitary connections on $E$, that was induced by the $\Si$-action on $X$ (Proposition \ref{Si_action_on_conn}). That $\Si$-action on $\cA_E$ was compatible with the Hamiltonian action of the gauge group $\cG_E$ in the sense of Property \eqref{comp_condition_group_actions} and with the momentum map $F$ of that action. Back in Section \ref{gauge_theoretic_picture}, this was only used to deduce that there was an induced $\Si$-action on $F^{-1}(\{0\})_\irr/\cG_E$, or more generally on $F^{-1}(\{\frac{d}{r}\})/\cG_E$ (Proposition \ref{Si_action_on_Hamil_quotient}). One can be slightly more general if one allows holomorphic principal $G_\C$-bundles, for $G_\C$ an arbitrary reductive complex Lie group, into the picture. Indeed, let us denote by $\theta$ the split real form of $G_\C$, by $K$ a $\theta$-invariant maximal compact subgroup of $G_\C$, and by $P$ a principal $K$-bundle. Recall that $\Si$ acts on $K$ via the homomorphism $\eps:\Si\lra\Z/2\Z$ of \eqref{construction_eps} and the involution $\theta$ (as in \eqref{Si_action_on_G_def}).

\begin{definition}\label{c_equiv_ppal_bdle}
Given a $2$-cocycle $c\in Z^2(\Si;\cZ(K))$, a $c$-equivariant structure on $P$ is a family $\tau:=(\tau_\si)_{\si\in\Si}$ of (normalised) bundle maps from $P$ to itself satisfying the conditions:
\begin{enumerate}
\item For all $\si\in\Si$, we have a commutative diagram $$\begin{CD}
P @>{\tau_\si}>> P \\
@VVV @VVV\\
X @>{\si}>> X
\end{CD}$$
\item For all $\si\in\Si$, all $p\in P$ and all $k\in K$, one has $\tau_\si(k\cdot p) = \theta_\si(k)\cdot \tau_\si(p)$.
\item $\tau_{1_\Si}=\Id_P$ and, for all $(\si_1,\si_2)\in\Si\times \Si$, $\tau_{\si_1} \tau_{\si_2} = c(\si_1,\si_2) \tau_{\si_1\si_2}$.
\end{enumerate}
Equivalently, we can think of $(\tau_\si)_{\si\in\Si}$ as a family $(\phi_\si:\si(P)\lra P)_{\si\in\Si}$ of principal $K$-bundles isomorphisms satisfying $\phi_{\si_1}\si_1(\phi_{\si_2}) = c(\si_1,\si_2)\phi_{\si_1\si_2}$. When $c\equiv1$, a $c$-equivariant structure $\tau$ will simply be called an equivariant structure.

\noindent A homomorphism between two $c$-equivariant bundles $(P,(\tau_\si)_{\si\in\Si})$ and $(P',(\tau'_\si)_{\si\in\Si})$ (for the same cocycle $c$) is a homomorphism of principal bundles $f:P\lra P'$ over $X$ such that, for all $\si\in \Si$, $f\circ\tau_\si= \tau'_\si\circ f$.
\end{definition}

For instance, the category of $c$-equivariant principal $G_\C$-bundles for the group $G_\C=\GL(r;\C)$, $\theta(g)=\ov{g}$, $\Si\simeq\Z/2\Z$ acting on $X$ via an anti-holomorphic involution and $c\equiv 1$ is equivalent to the category of rank $r$ Real vector bundles in the sense of Atiyah (\cite{Atiyah_reality}), because of the compatibilty relation $\ov{gv} = \ov{g}\,\ov{v}$ for the standard $\GL(r;\C)$-action on $\C^r$. 

Under the assumption of Definition \ref{c_equiv_ppal_bdle}, there is an induced $\Si$-action $(\beta^{\tau}_\si)_{\si\in\Si}$ on the space $\cA_P$ of $K$-connections on $P$, defined as in \eqref{Si_action_on_conn_def}. There is also a $\Si$-action $(\alpha^{\tau}_\si)_{\si\in\Si}$ on the gauge group $\cG_P$, which is compatible with $(\beta^{\tau}_\si)_{\si\in\Si}$ in the sense of \eqref{comp_condition_group_actions}. The assumption that $\theta$ is induced on $K$ by the split real form of $G_\C$ is then sufficient to guarantee that the $\Si$-action on $\cA_P$ preserves the set of projectively flat connections (which was for instance not always true for the real form $\theta(g)=(\ov{g}^t)^{-1}$of $G_\C=\GL(r;\C)$, as we saw in Section \ref{alg_geo_pic}). In what follows, we consider an element $\xi$ in the centre of the Lie algebra of $\cG_P$ such that the set $F^{-1}(\{\xi\})$ of projectively flat connections of curvature $\xi$ is non-empty. That set carries an induced $\Si$-action and, because of Relation \eqref{comp_condition_group_actions}, the fixed-point group $\cG_P^{(\Si,\tau)}$ acts on $\cA_P^{(\Si,\tau)}$, preserving $F^{-1}(\{\xi\})_\irr^\Si$. Note that the $\Si$-actions on $\cG_P$ and $\cA_P$ depend on the choice of the $c$-equivariant structure $\tau$ on $P$, as reflected in our notation. We can then form the quotient $$F^{-1}(\{\xi\})_\irr^{(\Si,\tau)} / \cG_P^{(\Si,\tau)}$$ and we have that the inclusion map $F^{-1}(\{\xi\})^{(\Si,\tau)}\hookrightarrow F^{-1}(\{\xi\})$ induces a map \begin{equation}\label{moduli_spaces_of_inv_conn}
j_{\tau}: F^{-1}(\{\xi\})_\irr^{(\Si,\tau)} / \cG_P^{(\Si,\tau)} \lra \left(F^{-1}(\{\xi\}) / \cG_P\right)^\Si\ ,
\end{equation} taking the $\cG_P^{(\Si,\tau)}$-orbit of a $\Si$-invariant connection to its $\cG_P$-orbit. As we saw in Proposition \ref{induced_action_is_canonical}, the target space of the map $j_{\tau}$ is independent of the choice of the $2$-cocycle $c\in Z^2(\Si;\cZ(K))$ and of the $c$-equivariant structure $(\tau_\si)_{\si\in\Si}$ on $P$. The source space, however, very much depends on those choices and the rest of this section is devoted to the study of the map $j_{\tau}$, namely to understanding its fibres and its image (we will show that $j_{\tau}$ is neither injective nor surjective in general). In Section \ref{moduli_of_equiv_bdles}, we will provide a modular interpretation of points in the quotient $F^{-1}(\{\xi\})_\irr^\Si / \cG_P^{(\Si,\tau)}$.

The fibres of $j_{\tau}$ are relatively easy to understand. Let $A_1$ and $A_2$ be two $\Si$-fixed connections in $F^{-1}(\{\xi\})_\irr$ such that $A_1$ and $A_2$ lie in the same $\cG_P$-orbit, i.e.\ $j_{\tau}(\cG_P^{(\Si,\tau)}\cdot A_1) = j_{\tau}(\cG_P^{(\Si,\tau)}\cdot A_2)$. Then there is an element $g\in \cG_P$ such that $g\cdot A_2 = A_1$. So, for all $\si\in \Si$, $$g^{-1}\cdot A_1 = A_2 = \beta^{\tau}_\si(A_2) = \beta^{\tau}_\si(g^{-1}\cdot A_1) = \alpha^{\tau}_\si(g^{-1})\cdot \beta^{\tau}_\si(A_1) = \alpha^{\tau}_\si(g^{-1})\cdot A_1,$$ which implies that $g\alpha^{\tau}_\si(g^{-1})\in \mathrm{Stab}_{\cG_P}(A_1)=\cZ(\cG_P)$. The map $\si \lmt g\alpha^{\tau}_\si(g^{-1})$ thus defined is readily seen to be a (normalised) $\cZ(\cG_P)$-valued $1$-cocycle on $\Si$, whose cohomology class is independent of the choice of the element $g$. Moreover, by definition, that cocycle splits over $G_\C$. Thus, we have defined a map $$f:j_{\tau}^{-1}(\{\cG_P\cdot A_1\}) \lra \ker \big(H^1(\Si;\cZ(\cG_P)) \lra H^1(\Si;\cG_P)\big).$$

\begin{proposition}\label{fibre_j_c_tau}
The non-empty fibres of $j_{\tau}$ are in bijection with the pointed set $$\ker \big(H^1(\Si;\cZ(\cG_P)) \lra H^1(\Si;\cG_P)\big).$$
\end{proposition}

\begin{proof}
It suffices to check that the map $f$ defined above is bijective. The surjectivity follows from the construction of the map $f$ and the definition of $$\ker \big(H^1(\Si;\cZ(\cG_P)) \lra H^1(\Si;\cG_P)\big).$$ As for the injectivity, if the $\cZ(\cG_P)$-valued $1$-cocycle $g\alpha^{\tau}_\si(g^{-1})$ actually splits over $\cZ(\cG_P)$, i.e.\ if there exists an element $a\in\cZ(\cG_P)$ such that, for all $\si\in\Si$, $g\alpha^{\tau}_\si(g^{-1}) = a\alpha^{\tau}_\si(a^{-1})$, then $\alpha^{\tau}_\si(a^{-1}g)= a^{-1}g$ and $(a^{-1}g)^{-1}\cdot A_1 = g^{-1}\cdot (a^{-1}\cdot A_1) = g^{-1}\cdot A_1= A_2$, since $\cZ(\cG_P)$ acts trivially on $\cA_P$. This proves that $A_1$ and $A_2$ lie in the same $\cG_P^{(\Si,\tau)}$-orbit, hereby showing that $f$ is injective.
\end{proof}

\noindent We note that nothing in the above uses that $\mathrm{Stab}_{\cG_P}(A)=\cZ(\cG_P)$ and works in full generality: we have proved that the fibre of $j_{\tau}$ above the point $j_{\tau}(\cG_P^{(\Si,\tau)}\cdot A)$ is in bijection with $$\ker \big(H^1(\Si;\mathrm{Stab}_{\cG_P}(A)) \lra H^1(\Si;\cG_P)\big).$$ Next, to study the image of $j_{\tau}$ for all $c\in Z^2(\Si;\cZ(K))$, let us introduce a map $$\cT: \left(F^{-1}(\xi\})_\irr / \cG_P\right)^\Si \lra H^2(\Si;\cZ(K)),$$ that we will call the \textit{type map}. This will explicitly use the fact that $\mathrm{Stab}_{\cG_P}(A)=\cZ(\cG_P)$, if only for the target space of the type map to make sense. In order to define the map $\cT$, we start with an actual equivariant structure on $P$, i.e.\ we assume that there is given a family $(\tau_\si)_{\si\in\Si}$ satisfying $\tau_{\si_1\si_2}=\tau_{\si_1}\tau_{\si_2}$ on $P$. This enables us to identify $\cZ(\cG_P)$ with $\cZ(K)$ $\Si$-equivariantly (with respect to the $\Si$-action $(\alpha^{\tau}_\si)_{\si\in\Si}$ on $\cG_P$ and the $\Si$-action $(\theta_\si)_{\si\in\Si}$ on $\cZ(K)$). Then we let $A_0$ be an irreducible, projectively flat connection on $P$, whose $\cG_P$-orbit is $\Si$-fixed in the quotient $F^{-1}(\{\xi\})/\cG_P$. This means that, for all $\si\in\Si$, there exists $u_\si\in \cG_P$ such that $u_\si\cdot\beta^{\tau}_\si(A_0)=A_0$ (since $\beta^{\tau}_{1_\Si}=\mathrm{Id}_{\cA_P}$, we can take $u_{1_\Si}=1_{\cG_P}$). Hence, on the one hand, $\beta^{\tau}_{\si_1\si_2}(A_0) = u_{\si_1\si_2}^{-1}\cdot A_0$ and, on the other hand, $$\beta^{\tau}_{\si_1\si_2}(A_0) = \beta^{\tau}_{\si_1}(\beta^{\tau}_{\si_2}(A_0)) = \beta^{\tau}_{\si_1}(u_{\si_2}^{-1}\cdot A_0) = \alpha^{\tau}_{\si_1}(u_{\si_2}^{-1})\cdot \beta^{\tau}_{\si_1}(A_0)= \alpha^{\tau}_{\si_1}(u_{\si_2}^{-1})u_{\si_1}^{-1}\cdot A_0.$$ So, for all $(\si_1,\si_2)\in \Si\times \Si$, the element $c_u(\si_1,\si_2) := u_{\si_1} \alpha^{\tau}_{\si_1}(u_{\si_2})u_{\si_1\si_2}^{-1}$ lies in $\mathrm{Stab}_{\cG_P}(A_0)=\cZ(\cG_P)$. The map $c:\Si\times\Si\lra \cZ(\cG_P)$ thus defined is readily seen to be a (normalised) $2$-cocycle, whose cohomology class in $H^2(\Si;\cZ(\cG_P))$ only depends on the $\Si$-invariant orbit $[A_0]:=\cG_P\cdot A_0$, not the connection $A_0$ or the elements $(u_\si)_{\si\in\Si}$. Since $\cZ(\cG_P)$ is $\Si$-equivariantly isomorphic to $\cZ(K)$, we have $H^2(\Si;\cZ(\cG_P)) \simeq H^2(\Si;\cZ(K))$ and the above construction defines a map \begin{equation}\label{type_map_def}\cT:\begin{array}{ccc}
(F^{-1}(\{\xi\})_\irr / \cG_P)^\Si & \lra & H^2(\Si;\cZ(K))\\
\left[A_0\right] & \lmt & [c]
\end{array}\end{equation} where $c(\si_1,\si_2)=u_{\si_1}\alpha^{\tau}_{\si_1}(u_{\si_2}) u_{\si_1\si_2}^{-1}$ is defined from any family $(u_\si)_{\si\in\Si}$ satisfying, for all $\si\in\Si$, $\beta^{\tau}_\si(A_0)=u_\si^{-1}\cdot A_0$. Note that, if the $\cG_P$-orbit of $A_0$ contains a connection $A_0'$ which is actually $\Si$-fixed, then $\cT([A_0])=[1]$ in $H^2(\Si;\cZ(K))$. Indeed, for all $\si\in\Si$, $\beta^{\tau}_\si(A_0') = A_0'$, so we can take $u'_{\si}=1$ and we get that $c(\si_1,\si_2) =1$. Therefore, when $c\equiv1$ and an equivariant structure $\tau=(\tau_\si)_{\si\in\Si}$ has been chosen on $P$, the map $$j_{\tau}: F^{-1}(\{\xi\})_\irr^\Si / \cG_P^{(\Si,\tau)} \lra \left(F^{-1}(\{\xi\}) / \cG_P\right)^\Si,$$ where the $\Si$-action on the left-hand side is induced by the equivariant structure $\tau$ as in \eqref{Si_action_on_conn_def}, satisfies \begin{equation}\label{map_to_the_trivial_fibre_of_the_type_map}\mathrm{Im}\,j_{\tau} \subset \cT^{-1}(\{[1]\}).\end{equation} 
\begin{remark}
By construction, the type map $$\cT:(F^{-1}(\{\xi\})/\cG_P)^{\Si} \lra H^2(\Si;\cZ(\cG_P))$$ lifts to a map $\cT':(F^{-1}(\{\xi\})/\cG_P)^{\Si} \lra H^1(\Si;\cG_P/\cZ(\cG_P))$, i.e.\ factors through the connecting homomorphism $\delta: H^1(\Si;\cG_P/\cZ(\cG_P)) \lra H^2(\Si;\cZ(\cG_P))$ induced by the short exact sequence of group homomorphisms $$1\lra \cZ(\cG_P) \lra \cG_P \lra \cG_P/\cZ(\cG_P)\lra 1.$$
\end{remark}
Two natural questions then are:
\begin{itemize}
\item Is $\mathrm{Im}\,j_{\tau}$ equal to $\cT^{-1}(\{[1]\})$?
\item What can one say about the other fibres of the type map $\cT$?
\end{itemize} The answer to both these questions is provided by the next construction. Let $\lambda\in H^2(\Si;\cZ(K))$ and assume that $\lambda$ lies in the image of the map $\cT$. This means that there exists a $\Si$-invariant $\cG_P$-orbit $[A_0]$ and a family $(u_\si)_{\si\in\Si}$ of elements of $\cG_P$ such that, for all $\si\in \Si$, $u_\si\cdot\beta^{\tau}_\si(A_0) = A_0$ and, moreover, the $2$-cocycle on $\Si$ defined by $c_u(\si_1,\si_2):= u_{\si_1}\alpha^{\tau}_{\si_1}(u_{\si_2}) u_{\si_1\si_2}^{-1}$ has cohomology class $\lambda$. Let us then set, for all $\si\in \Si$, $\phi'_\si:=u_\si\phi_\si:\si(P)\lra P$, $$\beta^{\tau,u}_\si: \begin{array}{rcl} \cA_P & \lra & \cA_P\\ A & \lmt & u_\si\beta^{\tau}_\si (A)\end{array}\quad \mathrm{and}\quad \alpha^{\tau,u}_\si: \begin{array}{rcl} \cG_P & \lra &\cG_P \\ g & \lmt & u_\si \alpha^{\tau}_\si(g) u_\si^{-1}\end{array}\ .$$ Then, for all $(\si_1,\si_2)\in\Si\times\Si$, \begin{eqnarray*} \phi'_{\si_1} \si_1(\phi'_{\si_2}) & = & u_{\si_1} \phi_{\si_1} \si_1(u_{\si_2}\phi_{\si_2})\\
& = & u_{\si_1} (\phi_{\si_1}\si_1(u_{\si_2}) \phi_{\si_1}^{-1}) \phi_{\si_1} \si_1(\phi_{\si_2})\\
& = & \big(u_{\si_1}\alpha^{\tau}_{\si_1}(u_{\si_2})\big) \big(\phi_{\si_1} \si_1(\phi_{\si_2})\big) \\
& = & \big(c_u(\si_1,\si_2) u_{\si_1}u_{\si_2}\big) \phi_{\si_1\si_2}\\
& = & c_u(\si_1,\si_2) \phi'_{\si_1\si_2}
\end{eqnarray*} i.e.\ the family $(\phi'_\si)_{\si\in\Si}$ defines a $c_u$-equivariant structure on $P$ (in the sense of Definition \ref{c_equiv_ppal_bdle}). Moreover, we have new $\Si$-actions on $\cA_P$ and $\cG_P$, respectively given by \begin{equation}\label{new_action_on_conn}\beta^{\tau,u}_\si(A) = \big((\phi'_\si)^{-1}\big)^*\si(A) = \big((u_\si\phi_\si)^{-1}\big)^* \si(A) = (u_\si^{-1})^*(\phi_\si^{-1})^*\si(A) = u_\si\cdot \beta^{\tau}_\si(A)\end{equation} and \begin{equation}\label{new_action_on_gauge_gp}\alpha^{\tau,u}_\si(A) = \phi'_\si\si(g)(\phi'_\si)^{-1} = u_\si\big(\phi_\si\si(g)\phi_\si^{-1}\big)u_\si^{-1} = u_\si\alpha^{\tau}_\si(g) u_\si^{-1}=u_\si \tau_\si g \tau_\si^{-1}u_\si^{-1}.\end{equation} We then have, for the connection $A_0$ we started with, $$\beta^{\tau,u}_\si(A_0) = u_\si \cdot \beta^{\tau}_\si(A_0) = u_\si \cdot (u_\si^{-1}\cdot A_0) = A_0,$$ i.e.\ $A_0$ is fixed for this new action, which we denote by $$A_0\in \cA_P^{(\Si,\tau,u)}.$$ The inclusion $F^{-1}(\{\xi\})_{\mathrm{irr}}^{(\Si,\tau,u)} \hookrightarrow F^{-1}(\{\xi\})_{\mathrm{irr}}$ induces a map 
\begin{equation}\label{moduli_spaces_of_inv_conn_modif_fmly}
j_{\tau,u} : F^{-1}(\{\xi\})_{\mathrm{irr}}^{(\Si,\tau,u)} / \cG_P^{(\Si,\tau,u)} \lra \big(F^{-1}(\{\xi\})_{\mathrm{irr}} / \cG_P\big)^{\Si}
\end{equation} which generalises the map $j_{\tau}$ introduced in \eqref{moduli_spaces_of_inv_conn} and where $\Si$ now acts on $\cG_P$ via the $\alpha^{\tau,u}_\si$ above.

As we can see, the new $\Si$-actions on $\cA_P$ and $\cG_P$, introduced respectively in \eqref{new_action_on_conn} and  \eqref{new_action_on_gauge_gp}, are entirely determined by the choice of a family $u=(u_\si)_{\si\in\Si}$ such that $c_u(\si_1,\si_2):= u_{\si_1}\alpha^\tau_{\si_1}(u_{\si_2}) u_{\si_1\si_2}^{-1}$ is a $\cZ(\cG_P)$-valued $2$-cocycle. These families already appeared to define the type map in \eqref{type_map_def}, so we give them a name.

\begin{definition}[Modifying family]\label{modif_fmly_def}
A modifying family of elements in $\cG_P$ is a tuple $u:=(u_\si)_{\si\in\Si}$ 
of elements $u_\sigma \in \cG_P$ indexed by $\si\in \Si$ such that
\begin{enumerate}
\item $u_{1_{\Si}} = 1_{\cG_P}$,
\item For $\si_1,\si_2 \in  \Si$, the elements $c_u(\si_1,\si_2):=u_{\si_1} \Psi_{\si_1}(u_{\si_2}) u_{\si_1\si_2}^{-1} \in \cG_P$ is $\cZ(\cG_P)$-valued. In this case, $c_u$ is actually a $2$-cocycle of $\Si$.
\end{enumerate}
\end{definition}

\noindent As a direct generalisation of \eqref{map_to_the_trivial_fibre_of_the_type_map}, we have \begin{equation*} \mathrm{Im}\,j_{\tau,u} \subset \cT^{-1}(\{[c_u]\})\end{equation*} for any modifying family with associated $2$-cocycle $c_u$. The subtle point is that the converse inclusion is not true in general, because we may also find, in that fibre, elements of $\mathrm{Im}\,j_{\tau'}$, where $\tau'$ is another $c$-equivariant structure on $P$, non-isomorphic to the previous one (this is what happens, for instance, for Real vector bundles of different topological types in \cite{Sch_JSG}) and we now analyse this phenomenon (the next paragraph runs parallel to \cite[Section3.3]{HS_quiver_autos}, where the analogous situation in the context of quiver representations is studied). 

Let us start from a normalised $2$-cocycle $c:\Si\times\Si\lra \cG_P$. Saying that the cohomology class $[c]$ lies in the image of the type map means that there exists a modifying family $u=(u_\si)_{\si\in\Si}$ whose associated $2$-cocycle $c_u$ is cohomologous to $c$. Let us then denote by $Z_u^1(\Si,\cG_P)$ the set of $\cG_P$-valued normalised $1$-cocycles with respect to the modified $\Si$-action \eqref{new_action_on_gauge_gp} on $\cG_P$. The following two lemmas have very simple proofs (see \cite[Lemmas 3.18, 3.19]{HS_quiver_autos}).

\begin{lemma}\label{equal_cocycles}
Replacing $u_\si$ with $a_\si u_\si$ where $a_\si\in\cZ(\cG_P)$ if necessary, we can assume that $c_u(\si_1,\si_2) = c(\si_1,\si_2)$.
\end{lemma}
 
\begin{lemma}\label{mod_family_and_1_cocycles}
There is a bijection $u'\lmt b_{u'}(\si):= u'_\si u_\si^{-1}$, between the set of all modifying families $u'$ such that $c_{u'}=c_u$ and the set $Z^1_u(\Si,\cG_P)$, whose inverse is the map $b\lmt u^b_{\si}:= b(\si)u_\si$. Moreover, two $1$-cocycles $b_{u'}$ and $b_{u''}$ are cohomologous in $Z^1_u(\Si;\cG_P)$ if and only if there exists $g\in\cG_P$ such that, for all $\si\in\Si$, $u''_\si=g u'_\si \alpha^\tau_{\si}(g^{-1})$.
\end{lemma}

\begin{remark}\label{cohom_implies_same_images}
If $b_1$ and $b_1$ are cohomologous $1$-cocycles in $Z^1_u(\Si;\cG_P)$, the maps $j_{\tau,u^{b_1}}$ and $j_{\tau,u^{b_2}}$ defined as in \eqref{moduli_spaces_of_inv_conn_modif_fmly} have the same images. Indeed, any $g$ satisfying $u^{b_2}_\si = g u^{b_1}_\si \alpha^\tau_\si(g^{-1})$ for all $\si\in\Si$ provides a group isomorphism $$\cG_P^{(\Si,\tau,u^{b_2})} = g\cG_P^{(\Si,\tau,u^{b_1})}g^{-1} \simeq \cG_P^{(\Si,\tau,u^{b_1})}$$ and an equivariant bijection $$F^{-1}(\{\xi\})_{\mathrm{irr}}^{(\Si,\tau,u^{b_2})} = g \cdot F^{-1}(\{\xi\})_{\mathrm{irr}}^{(\Si,\tau,u^{b_1})} \simeq F^{-1}(\{\xi\})_{\mathrm{irr}}^{(\Si,\tau,u^{b_1})}.$$
\end{remark}

\noindent We can now give a decomposition of the fibre of the type map $$\cT: \left(F^{-1}(\xi\})_\irr / \cG_P\right)^\Si\lra H^2(\Si;\cZ(K))$$ over $[c_u]$, where the indexing set is the orbit space 
$H^1_{u}(\Si;\cG_P)/H^1(\Si;\cZ(\cG_P))$: as $\cZ(\cG_P)$ is central and $\Si$-invariant in $\cG_P$, the map $(a,b)\lmt (a b)(\si) :=(a_\si b(\si))_{\si\in\Si}$ induces an action of the group $H^1_u(\Si,\cZ(\cG_P))$ on the set $H^1_u(\Si,\cG_P)$, where by $H^1_u$ we mean the first cohomology set associated to the $\Si$-action $\alpha^{\tau,u}$ on $\cG_P$ defined in \eqref{new_action_on_gauge_gp}; moreover, $H^1_{u}(\Si,\cZ(\cG_P))=H^1(\Si;\cZ(\cG_P))$ is actually independent of the choice of the family $u=(u_\si)_{\si\in\Si}$, as $\cZ(\cG_P)$ is the centre of $\cG_P$. For $[b]\in H^1_u(\Si,\cG_P)$, we shall denote its $H^1(\Si,\cZ(\cG_P))$-orbit by $\ov{[b]}$. As a side remark, this gives the following abstract classification result for $c_u$-equivariant structures on $E$.

\begin{proposition}\label{classif_of_c_equiv_structures}
Given a modifying family $u=(u_\si)_{\si\in\Si}$ with associated $2$-cocycle $c_u\in Z^2(\Si;\cZ(K))$ and an equivariant structure $(\tau_\si)_{\si\in\Si}$, the set of isomorphism classes of $c_u$-equivariant structures is in bijection with the pointed set $H^1_u(\Si;\cG_P)$.
\end{proposition}

\begin{proof}
First, recall from \eqref{Si_action_on_gauge_gp_def} that $\alpha^\tau_\si(g)=\tau_\si g\tau_\si^{-1}$. Moreover, $\tau_{\si_1}\tau_{\si_2} = \tau_{\si_1\si_2}$, since by assumption $\tau$ is an equivariant structure (for the trivial cocycle). The bijection we are looking for is then induced by the map 
\begin{equation}\label{abstract_classif_result}
\begin{array}{ccc}
Z^1_u(\Si;\cG_P) & \lra & \{c_u\mathrm{-equivariant\ structures}\}\\
b & \lmt & \tau^b_\si:= b(\si) u_\si \tau_\si = u^b_\si\tau_\si
\end{array}.
\end{equation} The family $\tau'$ thus defined is indeed a $c_u$-equivariant structure, since 
\begin{eqnarray*}
\tau^b_{\si_1}\tau^b_{\si_2} & = & u'_{\si_1} \tau_{\si_1} u'_{\si_2}\tau_{\si_2} \\
& = & u'_{\si_1} \alpha^{\tau}_{\si_1}(u'_{\si_2}) \tau_{\si_1}\tau_{\si_2}\\
& = & c_u(\si_1,\si_2) u'_{\si_1\si_2} \tau_{\si_1\si_2} \\
& = & c_u(\si_1,\si_2) \tau^b_{\si_1\si_2}
\end{eqnarray*} and replacing $b$ by the cohomologous $1$-cocycle $b'(\si):=u b(\si) u_\si \alpha^{\tau}_\si(u^{-1}) g_{\si}^{-1}$ yields the conjugate $c_u$-equivariant structure $\tau^{b'}_{\si} = u\tau^b_\si u^{-1}$.
\end{proof}

\noindent We also note that replacing $b$ by $b'(\si):=a_\si b(\si)$ in \eqref{abstract_classif_result}, where $a\in Z^1(\Si;\cZ(\cG_P))$, yields $\tau^{b'}_{\si}=a_\si\tau^b_\si$ for all $\si\in\Si$. In other words, the map \eqref{abstract_classif_result} is $H^1(\Si;\cZ(\cG_P))$-equivariant: this implies that the classification result of $c$-equivariant structures given in Proposition \ref{classif_of_c_equiv_structures} is independent of the choice of the modifying family $u$ with associated cocyle $c_u=c$. The next theorem is the main result of the paper. We refer to \cite[Theorem 3.21]{HS_quiver_autos} for a similar statement in the context of quiver representations.

\begin{theorem}\label{main_result}
Let $u$ be a modifying family of elements in $\cG_P$ in the sense of Definition \ref{modif_fmly_def} and let $c_u$ be the associated $2$-cocycle. Let $\cT$ be the type map introduced in \eqref{type_map_def}. Then there is a decomposition 
\[ \cT^{-1}(\{[c_u]\}) = \bigsqcup_{\ov{[b]} \in H^1_{u}(\Si,\cG_P) / H^1(\Si,\cZ(\cG_P))} \mathrm{Im}\, j_{\tau,u^b}\] 
where $u^b$ is the modifying family determined by $u$ and a choice of 1-cocycle $b \in [b]$ and $j_{\tau,u^b}$ is the map defined in \eqref{moduli_spaces_of_inv_conn_modif_fmly}. As in Proposition \ref{fibre_j_c_tau}, non-empty fibres of $j_{\tau,u}$ are in bijection with the pointed set $$\ker \big(H^1(\Si;\cZ(K)) \lra H^1_u(\Si;\cG_P)\big).$$
\end{theorem}

\begin{proof}
For each $[b] \in H^1_{u}(\Si,\cG_P)$, we take a representative $b$ of $[b]$ and consider the map $j_{\tau,u^b}$ defined in \eqref{moduli_spaces_of_inv_conn_modif_fmly}; the image of this map does not depend on our choice of representative by Remark \ref{cohom_implies_same_images}.
 
To show that these images cover $\cT^{-1}([c_u])$, take $\cG_P\cdot A\in \cT^{-1}([c_u])$. By definition of the type map, there exists a modifying family $u'=(u'_\si)_{\si\in\Si}$ of elements in $\cG_P$ such that $A \in F^{-1}(\{\xi\})_{\mathrm{irr}}^{(\Si,\tau,u)}$ and $[c_{u'}]=[c_u]$. By Lemma \ref{equal_cocycles}, we can assume that $c_{u'}=c_u$ and, by Lemma \ref{mod_family_and_1_cocycles}, there exists $b \in Z^1_{u}(\Si,\cG_P)$ such that $u' = u^b$. Hence, $\cG_P \cdot A \in \mathrm{Im}\, j_{\tau,u^b}$.

To prove that this union is disjoint, suppose that $\cG_P \cdot A \in \mathrm{Im}\, j_{\tau,u^{b_1}} \cap \mathrm{Im}\, j_{\tau,u^{b_2}}$. Then there exist modifying families $u_i$, for $i=1,2$, such that 
\begin{enumerate}
\item $[b_{u_i}]=[b_i] \in H^1_{u}(\Si,\cG_P)$,
\item $A \in F^{-1}(\{\xi\})_{\mathrm{irr}}^{(\Si,\tau,u_i)}$ for $i=1,2$,
\item $c_{u_1} = c_{u_2} =c_u$.
\end{enumerate}
The only one of these assertions that is not clear is the final one, which follows from the fact that if $c_{u^{b_i}}=c_u$, for $i=1,2$, and if $[b] = [b']$, then $c_{u^b} =c_{u^{b'}}$. From (2), we deduce that $a_{\si}:=u_{2,\si} u_{1,\si}^{-1} \in \mathrm{Stab}_{\cG_P}(A) = \cZ(\cG_P)$, from which we conclude that $b_{u_2,\si} = a_{\si} b_{u_1,\si}$ for all $\si$, therefore that $[b_{u_1}]$ and $[b_{u_2}]$ lie in the same $H^1(\Si,\cZ(\cG_P))$-orbit in $H^1_{u}(\Si,\cG_P)$. This completes the proof.
\end{proof}

Thus, we can completely describe the fixed-point set of the $\Si$-action on the moduli space $\cM_X(r,d)$ of stable holomorphic vector bundles of rank $r$ and degree $d$ by using the following strategy:
\begin{itemize}
\item We define a type map $\cT: \cM_X(r,d)^\Si \lra H^2(\Si;\cZ(K))$ and we decompose $\cM_X(r,d)^\Si$ into the union of fibres of $\cT$;
\item We show that any non-empty fibre $\cT^{-1}(\{[c]\})$ decomposes into a disjoint union, indexed by the pointed set $H^1_{u}(\Si;\cG_E)/H^1(\Si;\cZ(\cG_E))$, where $E$ is the unitary gauge group of a smooth Hermitian vector bundle of rank $r$ and degree $d$, equipped with a $c$-equivariant structure $\tau$ and $u=(u_\si)_{\si\in\Si}$ satisfies $u_{\si_1}\alpha^{\tau}_{\si_1}(u_{\si_2}) u_{\si_1\si_2}^{-1} = c(\si_1,\si_2)$.
\item In the decomposition above, each of the pieces of $\cT^{-1}(\{[c]\})$ is the image of a map $$j_{\tau,u^b}: F^{-1}(\{\frac{d}{r}\})_{\mathrm{irr}}^{(\Si,\tau,u)} /\cG_E^\Si \lra \big(F^{-1}(\{\frac{d}{r}\})/\cG_E\big)^\Si$$ whose non-empty fibres are all in bijection with the kernel of the natural map $$H^1(\Si;\cZ(\U(r))) \lra H^1_u(\Si;\cG_E),$$ whose source space is in fact independent of the bundle $E$.
\end{itemize}

We hope to have made a convincing case that methods of group cohomology can help clarify an otherwise intricate situation (to wit, the study of the fixed-point set of a group action on a coarse moduli space of geometric objects), while at the same time providing sufficiently general results, to be used in a variety of concrete geometric situations. For example here, the cohomology groups appearing as indexing sets in Theorem \ref{main_result} can be computed using similar techniques in examples that look \textit{a priori} different (the parabolic vector bundles of \cite{Andersen-Grove} and the Real vector bundles of \cite{Sch_JSG}, to mention a few). In \cite{HS_quiver_autos,HS_rational_points}, we apply similar techniques to construct branes in hyper-k\"ahler quiver varieties and to study arithmetic aspects of quiver representations over a non-algebraically closed field.

\subsection{Moduli of equivariant bundles}\label{moduli_of_equiv_bdles}

In Section \ref{inv_conn_section}, we saw that, given a cocycle $c\in Z^2(\Si;\cZ(K))$ and a $c$-equivariant structure $\tau=(\tau_\si)_{\si\in\Si}$ on a principal $K$-bundle $P$ (for $K$ a compact Lie group equipped with the involution coming from the split real form of $G_\C:=K_\C$), there was a map $$j_{\tau} : F^{-1}(\{\xi\})_\irr^\Si / \cG_P^{(\Si,\tau)} \lra \left(F^{-1}(\{\xi\})_\irr / \cG_P\right)^\Si$$ whose image was contained in the fibre $\cT^{-1}(\{[c]\})$ of the type map defined in \eqref{type_map_def} and whose non-empty fibres were in bijection with the pointed set $$\ker \big(H^1(\Si;\cZ(\cG_P)) \lra H^1(\Si;\cG_P)\big).$$ The goal of the present section is to provide a modular interpretation of the quotients $F^{-1}(\{\xi\})_\irr^\Si/\cG_P^{(\Si,\tau)}$, consisting of $\cG_P^{(\Si,\tau)}$-orbits of $\Si$-invariant, projectively flat, irreducible $K$-connections on $P$. First we recall from \cite{Singer} that a $K$-connection on $P$ may be seen as a holomorphic structure on the bundle $P_\C:= P\times_K K_\C$, obtained by extending the structure group of $P$ to the complexification $K_\C=G_\C$ of $K$, and that isomorphism classes of such structures are precisely the $\cG_\C$-orbits of such $K$-connections, where $\cG_\C:=\cG_P\times_{\mathrm{Ad}} K_\C$ is the complexified gauge group, i.e.\ the automorphism group of $P_\C$, acting on $K$-connections through the previous identification with holomorphic structures (hereby extending the natural action of $\cG_P$ on the space $\cA_P$ of $K$-connections on $P$). Then, for a $K$-connection $A$ on $P$ to be $\Si$-invariant means that, with respect to the holomorphic structure induced by $A$, the map $$\tau_\si^\C:\begin{array}{ccc} P_\C & \lra & P_\C\\
\left[p,g\right] & \lmt & [\tau_\si(p),\theta_\si(g)]
\end{array}$$ is holomorphic if $\si:X\lra X$ is holomorphic and anti-holomorphic if $\si$ is anti-holomorphic on $X$. Thus, a $\Si$-invariant $K$-connection $A$ on $P$ defines a $c$-equivariant structure in the holomorphic sense on the holomorphic $G_\C$-bundle $P_\C$ (simply modify Definition \ref{c_equiv_ppal_bdle} by adding the requirement that $\tau_\si^\C:P_\C\lra P_\C$ is holomorphic if $\si:X\lra X$ is holomorphic and anti-holomorphic if $\si$ is anti-holomorphic). If $A$ is moreover irreducible and projectively flat, then the holomorphic bundle $P_\C$ is simple and stable in the sense of Ramanathan (\cite{Ramanathan,Ramanathan_Subramanian}). Therefore, we can  interpret the quotient $F^{-1}(\{\xi\})_\irr^\Si / \cG_P^{(\Si,\tau)}$ as a moduli space of principal holomorphic $G_\C$-bundles equipped with an additional datum, namely a $c$-equivariant structure in the holomorphic sense. Then $F^{-1}(\{\xi\})_\irr^\Si / \cG_P^{(\Si,\tau)}$ is the space of isomorphism classes of geometrically stable, simple, $c$-equivariant holomorphic structures on the smooth principal $G_\C$-bundle $P_\C$.

Since $\Si$ is a finite group acting on the affine space $\cA_P$ by affine transformations, the fixed-point set $\cA_P^\Si$ is non-empty and the issue is to know whether it contains irreducible, projectively flat connections (in order for the moduli space $F^{-1}(\{\xi\})_\irr^\Si / \cG_P^{(\Si,\tau)}$ to be non-empty). We shall now focus on the proof of that result in the vector bundle case ($K=\U(r)$ and $\theta(g)=\ov{g}$ on $G_\C=\GL(r;\C)$). To prepare for it, let us, for notational convenience, denote by $\cC$ the space $\cA_E$ of all unitary connections / holomorphic structures on the Hermitian vector bundle $E$ and let us introduce the spaces $\Css$ and $\Cs$, consisting respectively of semi-stable and stable holomorphic structures on $E$. The $\Si$-action on $\cC$ preserves $\Css$ and $\Cs$ and we wish to show that $\Cs^\Si$ is non-empty and connected because, using results of \cite{Sch_JDG} to establish the existence of a deformation retraction from $\Cs^\Si$ onto $F^{-1}(\{\frac{d}{r}\})_\irr^\Si$, this is equivalent to saying that $F^{-1}(\{\frac{d}{r}\})_\irr^\Si$ is non-empty and connected (the existence of the deformation retraction follows from the invariance of the Yang-Mills flow under the action of $\Si$, as in Section 3.3 of \cite{Sch_JDG}, and the convergence of that flow, \cite{Dask,Rade}). First, we observe that, for $r=1$, any holomorphic line bundle on $E$ is stable so $\Cs^\Si=\cC^\Si$, which is a non-empty affine subspace of $\cC$. In particular, $\Cs^\Si$ is connected if $r=1$. We henceforth assume that $r\geq 2$. Recall that the genus $g$ of $X$ is assumed to greater or equal to $2$ and that we have a group homomorphism $\eps:\Si\lra \Z/2\Z$ (defined in \eqref{construction_eps}), whose kernel we shall denote by $\Si^+$. Consider then the Harder-Narasimhan stratification of $\cC$ (\cite{HN}): $\cC$ is the disjoint union of the strata $\cC_\mu$ where $\mu=((r_i,d_i))_{1\leq i\leq \ell}$ represents the rank and degree of the successive quotients of the Harder-Narasimhan filtration $0=\cE_0\subset \cE_1 \subset\, \ldots\, \subset \cE_\ell=\cE,$ which by construction is subject to the condition $\frac{d_1}{r_1} >\, \ldots\, > \frac{d_l}{r_l}$. The complex codimension $d_\mu$ of $\cC_\mu$ in $\cC$ is finite and given by $$d_\mu = \sum_{1\leq i <j \leq \ell} (d_ir_j - d_jr_i + r_ir_j(g-1)).$$ An element $\si\in\Si$ sends the Harder-Narasimhan filtration of $\cE$ to the Harder-Narasimhan filtration of $\si(\cE)$ so the $\Si$-action on $\cC$ preserves each stratum $\cC_\mu$ and we have that:
\begin{enumerate}
\item For each $\mu$, $\cC^{\Si^+}/\Cmu^{\Si^+}$ is a finite-dimensional complex vector space and $$\codim_{\cC^{\Si^+}}\,\Cmu^{\Si^+} = \codim_{\cC}\, \Cmu = \dmu.$$
\item The group $\Si/\Si^+$ acts on $\Cmu^{\Si^+}$ and, if that group is non-trivial, the real codimension of $\Cmu^{\Si}$ in $\cC^\Si$ is equal to $\dmu$, the complex codimension of $\Cmu^{\Si^+}$ in $\cC^{\Si^+}$.
\end{enumerate}
\noindent Let us now show that $\Css^{\Si^+}$ is non-empty and connected. For all $\mu$ such that $\Cmu^{\Si^+}\neq \Css^{\Si^+}$, $\Cmu^{\Si^+}$ is contained in the closure of $\cC_{\mu'}^{\Si^+}$ for some $\mu'$ of the form $((r_1,d_1),(r_2,d_2))$ with $\frac{d_1}{r_1} > \frac{d_2}{r_2}$ (and $r_1+r_2=r$). Since $r\geq 2$ and $g\geq 2$, one has $$d_{\mu'} = d_1r_2 -d_2r_1+r_1r_2(g-1) \geq 1+(r-1)(g-1) \geq 2,$$ So the complement of $\Css^{\Si^+}$ in $\cC^{\Si^+}$ is a countable union of closed submanifolds of complex codimension greater or equal to $2$, which proves that $\Css^{\Si^+}$ is open and dense in $\cC^{\Si^+}$. In particular, it is non-empty. By results of Daskalopoulos and Uhlenbeck (\cite{Dask_Uhl}), the inclusion map $j:\Css^{\Si^+} \hookrightarrow \cC^{\Si^+}$ induces isomorphisms on homotopy groups $\pi_k(\Css^{\Si^+}) \overset{j_*}{\lra} \pi_k(\cC^{\Si^+})$ for all $k\leq (r-1)(g-1)+1$. This implies in particular that $\Css^{\Si^+}$ is connected for all $r\geq 2$ and all $g\geq 2$. Again by results of \cite{Dask_Uhl}, $\Cs^{\Si^+}$ is open and dense in $\Css^{\Si^+}$ (in particular, it is non-empty) and it is connected if $(r-1)(g-1)\geq 2$, i.e.\ if $r>2$ or $g>2$. 

The same type of argument works for $\Css^{\Si} = (\Css^{\Si^+})^{\Si/\Si^+}$ and $\Cs^{\Si} = (\Cs^{\Si^+})^{\Si/\Si^+}$ because, either $\Si/\Si^+$ is trivial and there is nothing to prove, or $\Si/\Si^+\simeq\Z/2\Z$ and the complex codimension $d_{\mu'}$ in the proof above is now a real codimension, so the inequality $d_{\mu'}\geq 2$ is enough to guarantee that $\Css^\Si$ is open and dense (in particular, non-empty) in $\cC^\Si$. Then the results of Daskalopoulos and Uhlenbeck can be adapted to this real setting (see \cite{BHH}) to show that $\Css^\Si$ is connected if $r\geq 2$ and $g\geq 2$, and that $\Cs^\Si$ is always non-empty, as well as connected if $r>2$ or $g>2$. We have thus proved the following result.

\begin{theorem}
Recall that $g\geq 2$, $r\geq 1$ and $d\in\Z$. Given a $2$-cocycle $c\in Z^2(\Si;\cZ(\U(r)))$ and a $c$-equivariant Hermitian vector bundle $(E,\tau)$ over $(X,\Si)$, the moduli space $\cM_X^{\Si}(r,d,\tau):= F^{-1}(\{\frac{d}{r}\})_\irr^{\Si} / \cG_E^\Si$ is non-empty and connected, except possibly if $g=2$, $r=2$ and $d$ is even.
\end{theorem}

\begin{proof}
The proof for $r=1$, $g>2$ or $g=2$ and $r>2$ was given above. For $g=2$, $r=2$ and $d$ odd, the result also holds because $\Cs=\Css$ in that case and we have seen that $\Css^\Si$ is always non-empty and connected if $g\geq 2$.
\end{proof}

\noindent We expect $\cM_X^{\Si}(2,2d,\tau)$ to also be non-empty and connected when $X$ is of genus $2$ but the proof above is inconclusive in that case.

\section{Representations of orbifold fundamental groups}\label{NS_and_orb}

\subsection{Narasimhan-Seshadri correspondences}\label{NS_correspondence}

One may view the Narasimhan-Seshadri correspondence as follows. Let $\cH$ be the Poincar\'e upper half-plane and let $\Ga\subset \PSL(2;\R) = \Aut(\cH)$ be a discrete subgroup. If $\Ga$ acts freely and cocompactly on $\cH$, then there is a homeomorphism between the space $\Hom(\Ga;\U(r))_\irr/\U(r)$ of (equivalence classes of) irreducible unitary representations of $\Ga$ and the moduli space $\cM_X(r,0)$ of stable holomorphic vector bundles of rank $r$ and degree $0$ over the compact Riemann surface $X:=\cH/\Ga$ (in particular, $\pi_1(X)\simeq \Ga$). More generally, Narasimhan and Seshadri proved the existence of a homeomorphism between $\cM_X(r,d)$ and the space of unitary representations of a certain central extension $\Ga_d$ of $\Ga$ by $\Z$ (the group $\Ga_d$ can be identified with the fundamental group of the unit circle bundle of an arbitrary line bundle of degree $d$ over $X$, see \cite{Furuta-Steer}). Note that $\Ga_d$ no longer acts effectively on $\cH$ here, because the natural map $\Ga_d\lra \Ga\subset\Aut(\cH)$ has kernel $\Z$, by construction. Generalising in another direction, one may ask what happens when there is only a normal subgroup of finite index $\Ga'\subset \Ga$ such that $\Ga'$ acts freely. The compact Riemann surface $Y:=\cH/\Ga$ then has a natural orbifold structure, obtained from the action of the finite group $\Si:=\Ga/\Ga'$ on the Riemann surface $X:=\cH/\Ga'$, and in fact $\piorb(Y) \simeq \Ga$. We will see shortly what becomes of the Narasimhan-Seshadri correspondence when we assume moreover that the group $\Ga'$ is still cocompact (such an assumption being justified by a standard application of Selberg's lemma: $\Ga$ contains a torsion-free normal subgroup of finite index $\Ga'$). Yet another type of generalisation is obtained by considering discrete subgroups of the full automorphism group $\Aut^\pm(\cH) \simeq \PGL(2;\R) \simeq \PSL(2;\R)\rtimes\Z/2\Z$ (where the action of $\Z/2\Z$ on $\PSL(2;\R)$ is given by conjugation by the matrix $\begin{pmatrix}-1 & 0 \\ \ \ 0 & 1 \end{pmatrix}$, corresponding to the anti-holomophic involution $z\lmt -\ov{z}$ of $\cH$).This means allowing $\Ga$ to contain anti-holomorphic transformations of $\cH$. If we consider the group homomorphism $\eta:\Ga\lra \{\pm1\}\simeq\Z/2\Z$ sending anti-holomorphic maps to $(-1)$ and set $\Ga':=\ker\eta$, then saying that $\eta$ is non-trivial amounts to saying that the compact Riemann surface $X:=\cH/\Ga'$ has a natural anti-holomorphic involution, coming from the induced $(\Ga/\Ga')$-action on $\cH/\Ga'$ (i.e., as an algebraic curve, $X$ is defined over the field of real numbers). The topological surface $Y:=\cH/\Ga$ then has a natural structure of Klein surface (dianalytic manifold of dimension $2$) and, again, we have an isomorphism $\piorb(Y)\simeq \Ga$. The difference with the previous case can be explained as follows: in both cases we have a short exact sequence \begin{equation}\label{ses_piorb}1\lra \pi_1(X) \lra \piorb(Y) \lra \Si \lra 1\end{equation} and a group homomorphism $\eps:\Si\lra\Z/2\Z$ induced by the group homomorphism $\Aut^\pm(\cH)\lra \{\pm1\}$ taking anti-holomorphic transformations to $(-1)$, but in the first case (i.e.\ when $\Ga\subset\Aut(\cH)$) the homomorphism $\eps$ is trivial, while it is non-trivial if $\Ga$ contains an antiholomorphic transformation.

In the remainder of the present section, we will explain what type of Narasimhan-Seshadri correspondence we obtain in this orbifold setting. Further generalisations are possible, for instance to orbifold fundamental groups of Real Seifert manifolds (\cite{Sch_JDG}) or to discrete subgroups $\Ga\subset\Aut^\pm(\cH)$ which only have finite covolume (the case where $\Ga\subset \Aut(\cH)$ acts freely and is of finite covolume first appeared in the Mehta-Seshadri theorem, see \cite{MS}; it has been vastly generalised in \cite{Boalch_parahoric} and in \cite{Balaji-Seshadri}). Combining the previous two settings, we now consider a cocompact Fuchsian group $\Ga\subset \PGL(2;\R)=\Aut^\pm(\cH)$. By Selberg's lemma, $\Ga^+:=\Ga\cap\PSL(2;\R)$ contains a torsion-free normal subgroup of finite index, say $\Ga'$. We shall denote by $\eta:\Ga\lra \Z/2\Z$ the group homomorphism taking anti-holomorphic transformations to $(-1)$. We also fix once and for all a $2$-cocycle $e\in Z^2(\Z/2\Z;\cZ(\U(r)))$, where $\Z/2\Z$ acts on $\cZ(\U(r))\simeq S^1$ via the involution $z\lmt\ov{z}$, induced by the involution $\theta:u\lmt\ov{u}$ on $\U(r)$. The associated extension of $\Z/2\Z$ by $\U(r)$ will be denoted by $\U(r)\times_e \Z/2\Z$. First, let us discuss an appropriate class of representations for orbifold fundamental groups of the type $\Ga$ above.

\begin{definition}\label{orbifold_rep}
An $e$-twisted unitary representation of $\Ga$ is a group homomorphism $\rho:\Ga\lra \U(r)\times_e\Z/2\Z$ such that the diagram $$\xymatrix{\Ga \ar[rr]^{\rho} \ar[rd]^{\eta} & & \U(r)\times_e \Z/2\Z \ar[dl] \\ & \Z/2\Z &}$$ commutes. The set of such representations will be denoted by $\Hom_\eta(\Ga;\U(r)\times_e\Z/2\Z)$. Two representations $\rho_1$, $\rho_2$ are called equivalent if there is an element $u\in\U(r)$ such that, for all $\ga\in\Ga$, $\rho_2(\ga) = u\rho_1(\ga)u^{-1}$.
\end{definition}

\noindent We note that, if $\eta$ is the trivial homomorphism (i.e.\ $\Ga\subset\PSL(2;\R)$), then $$\Hom_\eta(\Ga;\U(r)\times_e\Z/2\Z)=\Hom(\Ga;\U(r)).$$ We now want to interpret the representation space $$\Hom_\eta(\Ga;\U(r)\times_e\Z/2\Z)/\U(r)$$ as a moduli space of holomorphic vector bundles on a certain Riemann surface $X$. To that end, we observe the following: if we fix a cocompact normal subgroup of finite index $\Ga'\subset\Ga^+=\Ga\cap\PSL(2;\R)$ that acts freely on $\cH$ and set $\Si:=\Ga/\Ga'$, then we have a compact Riemann surface $X:=\cH/\Ga'$ equipped with an action of the finite group $\Si\hookrightarrow\Aut^\pm(X)$ and the group homomorphism $\eta:\piorb(X/\Si)\simeq \Ga \lra \Z/2\Z$ canonically lifts to a group homomorphism $\piorb(X/\Si)\lra \Si$, making the diagram $$\xymatrix{& \Si \ar[d]^{\eps} \\ \piorb(X/\Si) \ar[ur] \ar[r]^{\eta} & \Z/2\Z}$$ commute, where $\eps:\Si\lra \Z/2\Z$ is defined as in \eqref{construction_eps}. The set $\Hom_\eta(\Ga;\U(r)\times_e\Z/2\Z)$ can therefore also be seen as the set $\Hom_\eps(\Ga;\U(r)\times_e\Z/2\Z)$ consisting of group homomorphisms $\rho:\piorb(X/\Si)\lra\U(r)\times_e\Z/2\Z$ such that the diagram $$\begin{CD}
1 @>>> \pi_1(X) @>>> \piorb(X/\Si) @>>> \Si @>>> 1\\
@. @VVV @VV{\rho}V @VV{\eps}V @.\\
1 @>>> \U(r) @>>> \U(r)\times_e \Z/2\Z @>>> \Z/2\Z @>>> 1
\end{CD}$$ commutes. In particular, the set $\Hom_\eps(\piorb(X/\Si);\U(r)\times_e\Z/2\Z)$ that we have just defined is actually independent of the choice of the presentation $(X,\Si)$ for the orbifold $Y:=\cH/\Ga\simeq X/\Si$ (i.e.\ independent of the choice of $\Ga'$).

We will now show that, 
for all $\rho\in\Hom_\eps(\piorb(X/\Si);\U(r)\times_e\Z/2\Z)$ 
the polystable holomorphic vector bundle $$\cE_\rho:= (\cH\times\C^r)/\pi_1(X)$$ over $X$ comes equipped with a natural $(\eps^*e)$-equivariant structure in the sense of Definition \ref{c-equiv_struct}, where $\eps^*e
\in Z^2(\Si;\cZ(\U(r)))$. In order to see this, the basic construction goes as follows. We first choose a family $(\ga_\si)_{\si\in\Si}$ of elements of $\piorb(X/\Si)$ such that, for all $\si\in\Si$, $\ga_\si$ maps to $\si$ under the canonical map $\piorb(X/\Ga)\lra\Si$ constructed above (note that, in general, we do not have $\ga_{\si_1}\ga_{\si_2}=\ga_{\si_1\si_2}$; the short exact sequence \eqref{ses_piorb} need not admit splittings) and we consider the transformations $$\taut_{\si} : \begin{array}{ccc} \cH\times \C^r & \lra & \cH\times \C^r\\ (p,v) & \lmt & (\ga_\si\cdot p, \rho(\ga_\si)\cdot v)\end{array}.$$ Here, $\rho(\ga_\si)=(u_\si,\eps(\si))\in\U(r)\times_c\Z/2\Z$ acts on $v\in\C^r$ via $\rho(\ga_\si)\cdot v := u_\si\theta_\si(v)$. Each $\taut_\si$ then descends to a map $\tau_\si:\cE_\rho\lra\cE_\rho$ and the family $(\tau_\si)_{\si\in\Si}$ thus constructed satisfies the conditions of Definition \ref{c-equiv_struct} for the cocycle $c:=\eps^*e\in Z^2(\Si;\cZ(\U(r)))$. Moreover, the isomorphism class of the $(\eps^*e)$-equivariant bundle $(\cE_\rho,(\tau_\si)_{\si\in\Si})$ obtained in that way is independent of the choice of the family $(\ga_\si)_{\si\in\Si}$ and the isomorphism class of the representation $\rho\in\Hom_\eps(\piorb(X/\Si);\U(r)\times_e\Z/2\Z)$. We refer to Section 2.2 of \cite{Sch_JDG} for further details of that construction, where it is presented in the case where $\eps:\Si\lra\Z/2\Z$ is an isomorphism but with the same general formalism. We have therefore defined a Narasimhan-Seshadri map, from the representation space 
$$\Hom_\eps(\Ga;\U(r)\times_e \Z/2\Z) / \U(r),$$ to the space of isomorphism classes of polystable $(\eps^*e)$-equivariant holomorphic vector bundles of rank $r$ and degree $0$. As a matter of fact, there is such a Narasimhan-Seshadri map for every choice of a presentation $(X,\Si)$ for the orbifold $Y=\cH/\Ga$. Our main result in this section is then the following analogue of the Narasimhan-Seshadri theorem for orbifold fundamental groups of the type $\Ga\subset\PGL(2;\R)$ that we have been considering.

\begin{theorem}\label{NS_thm}
Let $\Ga\subset\PGL(2;\R)$ be a cocompact Fuchsian group and let $\Ga'$ be a torsion-free normal subgroup of finite index of the group $\Ga^+:=\Ga\cap\PSL(2;\R)$. Denote by $\Si$ the finite group $\Ga/\Ga'$, by $\eps:\Si\lra\Z/2\Z$ the group homomorphism constructed in \eqref{construction_eps}, and by $X$ the compact Riemann surface $\cH/\Ga'$. Then, given a $2$-cocycle $e\in Z^2(\Z/2\Z;\cZ(\U(r)))$, there is a homeomorphism between the representation space $$\Hom_\eps(\Ga;\U(r)\times_e\Z/2\Z)/\U(r)$$ and the moduli space $\ov{\cM}_X^\Si(r,0,\eps^*e)$, consisting of isomorphism classes of polystable $(\eps^*e)$-equivariant holomorphic vector bundles of rank $r$ and degree $0$, that moduli space being equipped with the topology given by the bijective correspondence $$\ov{\cM}_X^\Si(r,0,\eps^*e)\simeq \bigsqcup_{[\tau]} F^{-1}(\{0\})^\Si / \cG_E^\Si,$$ where the union runs over all smooth isomorphism classes of $(\eps^*e)$-equivariant Hermitian vector bundles $(E,\tau)$ of rank $r$ and degree $0$.
\end{theorem}

Apart from its somewhat intricate statement, this correspondence is perhaps not very satisfying yet, for the following two reasons. First, topological types of $(\eps^*e)$-equivariant Hermitian vector bundles are not know in general (see however Proposition \ref{classif_of_c_equiv_structures} for an indication of what to compute in general and \cite{BHH} for an explicit description in the case where $\eps:\Si\lra\Z/2\Z$ is an isomorphism). And second, there is no available algebraic construction of the "moduli space" $$\ov{\cM}_X^\Si(r,0,\eps^*e)=\bigsqcup_{[\tau]}\ov{\cM}_X^\Si(r,0,\tau)$$ that the author is aware of. In Section \ref{modular_interp_section}, we were at best able to relate, via the maps $j_{\eps^*e,\tau}$, the smaller moduli spaces $\cM_X^\Si(r,0,\tau)$, consisting of geometrically stable bundles only, to fixed points of the action of $\Si$ on the moduli variety $\cM_X(r,0)$. Providing an algebraic construction, for $\Ga$ satisfying the assumptions of Theorem \ref{NS_thm}, of moduli spaces of vector bundles over the orbifold $\cH/\Ga$  (in an appropiate category) and relating them to orbifold representations of $\Ga$ in the sense of Definition \ref{orbifold_rep} seems like a challenging problem, quite close in spirit to the original Narasimhan-Seshadri and Mehta-Seshadri correspondences. Note that the main difference between the problem just stated and the papers \cite{Boalch_parahoric,Balaji-Seshadri} is the fact that our group $\Ga$ is authorised to contain anti-holomorphic transformations of the Poincar\'e half-plane. We hope that, even after the problem above has been solved, the gauge-theoretic perspective developed in the present paper will be useful to study the topology (in particular, the number of connected components) of the moduli space $\Hom_{\eta}(\Ga;\U(r)\times_e\Z/2\Z) /\U(r)$, using an arbitrary presentation $(X,\Si)$ of the orbifold $Y=\cH/\Ga$, which was for instance the strategy adopted in \cite{Sch_JSG} for the case where $\eps:\Si\lra \Z/2\Z$ is an isomorphism. In fact, in the latter case, more topological information can be obtained by following up on this circle of ideas, such as Betti numbers with mod $2$ coefficients of the quotients $F^{-1}(\{\frac{d}{r}\})^\Si/\cG_E^\Si$, as shown in \cite{LS}.

\subsection{Holonomy representations associated to invariant connections}

To prove Theorem \ref{NS_thm}, it suffices to construct an inverse to the Narasimhan-Seshadri map constructed in Section \ref{NS_correspondence}. Such a map is provided, quite classically, by taking the holonomy of the relevant connections, namely here the $\Si$-invariant connections considered in Section \ref{inv_conn_section}. The main technical result is the following one, for the proof of which we refer to Proposition 4.2 of \cite{Sch_JDG}.

\begin{proposition}\label{par_transport_Si_inv_connection}
Let $(E,(\tau_\si)_{\si\in\Si})$ be an $(\eps^*e)$-equivariant Hermitian vector bundle over $(X,\Si)$ and let $A$ be a $\Si$-invariant unitary connection on $E$. Then, given any path $\ga\in X$, the parallel transport operators $T^A_\ga$ and $T^A_{\si\circ\ga}$ satisfy $T^A_{\si\circ\ga} = \tau_\si T^A_\ga \tau_{\si}^{-1}$. That is to say, one has a commutative diagram $$\begin{CD}
E_{\si(\ga(0))} @>{T^A_{\si\circ\ga}}>> E_{\si(\ga(1))}\\
@A{\tau_\si}AA @A{\tau_\si}AA\\
E_{\ga(0)} @>{T^A_{\ga}}> >E_{\ga(1)}.
\end{CD}$$
\end{proposition}

\noindent If moreover $A$ is a flat connection, then the parallel transport operators on loops at a given base point $x\in X$ induce a group homomorphism $\Hol:\pi_1(X;x)\lra \U(r)$, and the point is now to show that the $\Si$-invariance of $A$ implies that $\Hol$ extends to a group homomorphism $\widetilde{\Hol}: \piorb(X/\Si) \lra \U(r)\times_e \Z/2\Z$ that makes the following diagram commute $$\begin{CD}
1 @>>> \pi_1(X) @>>> \piorb(X/\Si) @>>> \Si @>>> 1\\
@. @VV{\Hol}V @VV{\widetilde{Hol}}V @VV{\eps}V @.\\
1 @>>> \U(r) @>>> \U(r)\times_e \Z/2\Z @>>> \Z/2\Z @>>> 1
\end{CD}$$ which follows from Proposition \ref{par_transport_Si_inv_connection} (we refer to Theorem 4.4 of \cite{Sch_JDG} for details). This completes the proof of Theorem \ref{NS_thm}.

\subsection{Equivariant representations}

We conclude by explaining how to simplify the statement of Theorem \ref{NS_thm} when the Riemann surface $X$ contains a $\Si$-fixed point $x$. Then the group $\Si$ acts on $\pi_1(X;x)$ via $\ga\lmt\si\circ\ga$ and we can consider the induced $\Si$-action on $\Hom(\pi_1(X;x);\U(r))$ given by $\beta^{\tau}_\si(\rho)=\theta_\si\circ\rho\circ \si^{-1}$. This is compatible with the $\Si$-action $(\theta_\si)_{\si\in\Si}$ on $\U(r)$, in the sense that $\beta^{\tau}_\si(g\rho g^{-1}) = \theta_\si(g)\beta^{\tau}_\si(\rho)\theta_\si(g)^{-1}$, and we note that $\rho$ is fixed by that $\Si$-action if and only if $\rho$ is $\Si$-equivariant. We are then in the same type of situation as in Section \ref{inv_conn_section}: the group $\U(r)^\Si$ acts on the set $\Hom(\pi_1(X;x);\U(r))_\irr^\Si$ and there is a map $$j: \Hom(\pi_1(X;x);\U(r))_\irr^\Si / \U(r)^\Si \lra (\Hom(\pi_1(X;x);\U(r))/\U(r))^\Si$$ whose image is contained in the fibre $\cT^{-1}(\{[1]\})$ of a type map $$\cT:  (\Hom(\pi_1(X;x);\U(r))/\U(r))^\Si \lra H^2(\Si;\cZ(\U(r)))$$  which is defined as in Section \ref{inv_conn_section}. The point is to realise that the choice of $x\in X^\Si$ defines an isomorphism $\piorb(X/\Si)\simeq \pi_1(X;x)\rtimes \Si$ for the action of $\Si$ on $\pi_1(X;x)$ described earlier. And then, $\Si$-equivariant representations $\rho:\pi_1(X;x)\lra \U(r)$ will extend to representations $\widehat{\rho}:\pi_1(X;x)\rtimes \Si \lra \U(r)\rtimes \Z/2\Z$ by setting $\widehat{\rho}(\ga,\si) := (\rho(\ga),\eps(\si))$. The map $$\begin{array}{ccc}\Hom(\pi_1(X;x);\U(r))^\Si / \U(r)^\Si & \lra & \Hom_\eps(\pi_1(X;x) \rtimes \Si; \U(r)\rtimes \Si) / \U(r)\\
\left[\rho\right] & \lmt & [\widehat{\rho}]
\end{array}$$ thus defined is always injective, because if $\rho_1,\rho_2:\pi_1(X;x)\lra \U(r)$ are two $\Si$-equivariant representations such that $\widehat{\rho_2} = (g,1)\cdot \widehat{\rho_1}$ for some $g\in \U(r)$, then, for all $\si\in\Si$, $$(1,\eps(\si)) = \widehat{\rho_2}(1,\si) = (g,1)\cdot\widehat{\rho_1}(1,\si) = (g,1)(1,\eps(\si))(g^{-1},1) = (g\theta_\si(g^{-1}),1)$$ so $g\theta_\si(g^{-1})=1$, which means that $g\in\U(r)^\Si$, hence that $\rho_1=\widehat{\rho_1}|_{\pi_1(X;x)}$ and $\rho_2=\widehat{\rho_2}|_{\pi_1(X;x)}$ are $\U(r)^\Si$-conjugate. Moreover, as shown in Proposition 2.9 of \cite{Sch_JDG}, the map $[\rho]\lmt[\widehat{\rho}]$ is surjective when $H^1(\Si;\U(r))=\{[1]\}$, and when $H^1(\Si;\U(r))$ is not trivial, it is possible to modify the $\Si$-action on $\Hom(\pi_1(X;x);\U(r))$ and on $\U(r)$ by introducing a cocycle $(a_\si)_{\si\in\Si}$ representing a given cohomology class $\kappa\in H^1(\Si;\U(r))$, as in Section \ref{inv_conn_section} of the present paper, to prove that all points in $\Hom_\eps(\pi_1(X;x) \rtimes \Si; \U(r)\rtimes \Si) / \U(r)$ come from $\Si$-equivariant representations, for different actions of $\Si$. 



\end{document}